\documentclass[11pt]{article}
\usepackage{amsthm,amsfonts,amsmath,amssymb,graphicx}
\usepackage{enumerate}
\usepackage{tabularx,booktabs}
\usepackage[left=2.2cm,right=2.2cm,top=2cm]{geometry}
\usepackage{subfig}
\usepackage{tikz}

\usetikzlibrary{arrows,shapes,positioning}
\usetikzlibrary{decorations.markings}
\tikzstyle arrowstyle=[scale=1]
\tikzstyle directed=[postaction={decorate,decoration={markings,
	mark=at position 0.5 with {\arrow[arrowstyle]{stealth}}}}]
\tikzstyle reverse directed=[postaction={decorate,decoration={
	markings,mark=at position 0.5 with {\arrowreversed[arrowstyle]
	{stealth}}}}]

\newtheorem{defn}{{\bf Definition}}[section]
\newtheorem{eg}[defn]{{\bf Example}}
\newtheorem{lemma}[defn]{{\bf Lemma}}
\newtheorem{prop}[defn]{{\bf Proposition}}
\newtheorem{theo}[defn]{{\bf Theorem}}
\newtheorem{cor}[defn]{{\bf Corollary}}

\newcommand{\Kfour}{{\mathcal{K}}(4)}
\newcommand{\Kd}{{\mathcal{K}}(d)}
\newcommand{\lKd}{\overline{\mathcal{K}}(d)}
\newcommand{\lKfive}{\overline{\mathcal{K}}(5)}
\newcommand{\skel}[2]{\mathrm{skel}_{#1}(#2)}
\newcommand{\lk}[2]{\mathrm{lk}_{#1}(#2)}
\newcommand{\Kdn}{{\mathcal K}^\ast(d)}
\newcommand{\lKdn}{{\overline{\mathcal K}}^{\,\ast}(d)}
\newcommand{\Kfourn}{{\mathcal K}^\ast(4)}
\newcommand{\lKfiven}{{\overline{\mathcal K}}^{\,\ast}(5)}

\newcommand{\st}[2]{\mathrm{st}_{#1}(#2)}
\newcommand{\ovr}{\overrightarrow}

\newcommand{\bkt}[1]{[#1]}

\newcommand{\TPSS}{S^{\hspace{.2mm}3}\! \times \hspace{-3.3mm}_{-} \,
S^{\hspace{.1mm}1}}

\newcommand{\TPSSDS}{S^{\hspace{.2mm}d-1}\! \times \hspace{-3.3mm}_{-} \, S^{\hspace{.1mm}1}}

\newcommand{\TPS}{S^{\hspace{.2mm}3}\! \times \hspace{-5.50mm}_{-} \, S^{\hspace{.1mm}1}}

\title{Minimal triangulations of $(S^3\times S^1)^{\#3}$ and
$(\TPS)^{\#3}$}
\author{\bf Nitin Singh \\
	{\small Department of Mathematics,
	Indian Institute of Science,
	Bangalore-560012, India.} \\
	{\small email: nitin@math.iisc.ernet.in}
}

\begin{document}
\maketitle

\abstract{A triangulated $d$-manifold $K$, satisfies the inequality
$\binom{f_0(K)-d-1}{2}\geq \binom{d+2}{2}\beta_1(K;\mathbb{Z}_2)$ for
$d\geq 3$. The triangulated
$d$-manifolds that meet the bound with equality are called {\em tight
neighborly}. In this paper, we present tight neighborly triangulations of $4$-manifolds
on $15$ vertices with $\mathbb{Z}_3$ as automorphism group. One such example was
constructed by Bagchi and Datta in 2011. We show that there are exactly 12
such triangulations up to isomorphism, 10 of which are orientable. 
}
\bigskip

\noindent{\small {\em MSC 2000\,:} 57Q15, 57R05.

}
\bigskip

\section{Introduction}
Minimal triangulations play an important role in combinatorial topology. In
this regard, lower bounds on the number of vertices needed in a
triangulation of a topological space, in terms of it's topological
invariants are particulary important. For compact surfaces, Heawood's
inequality is one such lower bound, which states that a surface with
Euler characteristic $\chi$, requires at least $\lceil \frac{1}{2}
(7+\sqrt{49-24\chi})\rceil$ vertices for it's triangulation. A similar
analogoue for higher dimensions was proved by Novik and Swartz in
\cite{ns}. They prove that a triangulation of any manifold $K$, of dimension
$d\geq 3$ satisfies,
\begin{equation}\label{eq:tn}
\binom{f_0(K)-d-1}{2}\geq \binom{d+2}{2}\beta_1(K;\mathbb{Z}_2).
\end{equation}
The triangulations which satisfy
(\ref{eq:tn}) with equality were called {\em tight neighborly} by Lutz
{\em et al.} in \cite{lss}. It must be noted that tight neighborly triangulations,
when they exist, are vertex minimal. Infact, for $d\geq 4$, they have
compontent-wise minimal $f$-vector. A trivial example of a tight neighborly triangulation
is the $(d+2)$-vertex sphere $S^d_{d+2}$ for $d\geq 3$. In 1986, K\"{u}hnel
constructed $2d+3$ vertex triangulations of $S^{\,d-1}\times S^1$ for even
$d$, and $\TPSSDS$ for odd $d$. K\"{u}hnel's triangulations are tight
neighborly with $\beta_1=1$. Until recently, very few examples of tight
neighborly triangulations apart from K\"{u}hnel's series were known. The triangulation of a $4$-manifold constructed by Bagchi and
Datta in \cite{bd10}, was a first sporadic example of a tight neighborly
triangulation. Very recently Datta and Singh \cite{bdnsinf}, have given another infinite
family of tight neighborly triangulations, exploiting some
unique combinatorial properties of such triangulations. These techniques
were further extended in \cite{nsnon} to show the non-existence of
tight neighborly triangulations for $(S^{\,d-1}\times S^1)^{\#2}$ and
$(\TPSSDS)^{\#2}$. In this paper, our aim is to employ the recently
developed combinatorial criteria in \cite{bdnsinf} and \cite{nsnon} to enumerate all
possible tight neighborly triangulations of $(\TPSS)^{\#3}$ and
$(S^{\,3}\times S^1)^{\#3}$, which have a non-trival $\mathbb{Z}_3$ action
as the Bagchi-Datta example. As a result, we obtain one more tight neighborly triangulation of $(\TPSS)^{\#3}$ and $10$ new tight neighborly
triangulations of $(S^{\,3}\times S^1)^{\#3}$. 

\section{Preliminaries}
\subsection{Triangulations}
All simplicial complexes considered here are finite and abstract. Members
of a simplicial complex are called {\em faces}. The empty set is a face of
every simplicial complex. A
simplicial complex is called {\em pure} if all it's maximal faces (called
{\em facets}) have the
same dimension. A pure
$d$-dimensional simplicial complex is called a {\em weak pseudomanifold}
without boundary (resp., with boundary), if all it's $d-1$ dimensional
faces occur in exactly two (resp., at most two) facets. With a weak
pseudomanifold $X$, we associate the graph $\Lambda(X)$, whose vertices are
facets of $X$, and two facets are adjacent in $\Lambda(X)$ if they
intersect in a face of co-dimension one. A weak pseudomanifold $X$ of
dimension $d$ is called a $d$-pseudomanifold, if $\Lambda(X)$ is a
connected graph. Triangulations of connected manifolds are naturally pseudomanifolds.

If $X$ is a $d$-dimensional simplicial complex then, for $0\leq j\leq d$,
the number of its $j$-faces are denoted by $f_j=f_j(X)$. The vector
$(f_0,\ldots,f_d)$ is called the {\em face vector} of $X$. For simplicial complexes $X$ and $Y$, we define the {\em join} of $X$ and $Y$, denoted by $X\ast Y$
as,
\[ X\ast Y := \{\alpha\sqcup \beta : \alpha\in X,\beta\in Y\}. \]
If $X$ and $Y$ are pseudomanifolds of dimension $m$ and $n$ respectively
then, $X\ast Y$ is a
pseudomanifold of dimension $m+n+1$. When $X$ consists of a single vertex
$x$, we denote the join as $x\ast Y$, and call the join as {\em cone} over
$Y$. 

For a vertex $x\in X$, we define the subcomplex $\lk{X}{x}$,
called the {\em link} of $x$ in $X$ as,
\[ 
\lk{X}{x} := \{\alpha\in X: \alpha\cup \{x\}\in X, x\not\in \alpha\}. 
\]
The cone $x\ast \lk{X}{x}$ is called the {\em star} of $x$ in
$X$ and is denoted by $\st{X}{x}$. By the $k$-skeleton of a simplicial
complex $X$, denoted by $\skel{k}{X}$, we shall mean the subcomplex of
$X$ consisting of faces of dimension at most $k$. We call the
$1$-skeleton of a simplicial complex as it's {\em edge graph}. When the
edge graph is a complete graph on the vertex set, we will call the complex
to be {\em neighborly}.

\subsection{Walkup's class $\Kd$}
A very natural class of triangulations of a topological ball is the class of {\em
stacked} balls. A {\em standard $d$-ball} is the pure $d$-dimensional
simplicial complex with one facet. A simplicial complex $X$ is called a {\em
stacked $d$-ball} if it is obtained from standard $d$-ball by successively
pasting $d$-simplices along $(d-1)$-faces (cf. \cite[Section 2]{bdnsinf}).
A {\em stacked} $d$-sphere is defined as the boundary of a stacked
$(d+1)$-ball. Clearly a stacked $d$-ball and a stacked $d$-sphere triangulate
the $d$-ball and $d$-sphere, respectively. Following result from
\cite{bdnsinf}, gives a combinatorial characterization of a stacked ball.
\begin{prop}\label{prop:bdns}
Let $X$ be a pure $d$-dimensional simplicial complex. 
\begin{enumerate}[{\rm (a)}]
\item If $\Lambda(X)$ is a tree then $f_0(X)\leq f_d(X)+d$.
\item $\Lambda(X)$ is a tree and $f_0(X)=f_d(X)+d$ if and only if $X$ is a
stacked $d$-ball.
\end{enumerate}
\end{prop}

In \cite{Walkup}, Walkup introduced the class $\Kd$ of simplicial
complexes, all whose vertex-links are {\em stacked spheres}. Clearly
members of $\Kd$ are triangulated $d$-manifolds. The
following result of Novik and Swartz shows that for dimensions four and
above, tight neighborly triangulations lie within the class $\Kd$. 
\begin{prop}[{Novik \& Swartz \cite{ns}}]\label{prop:ns}
For $d\geq 4$, if $M$ is tight neighborly, then $M$ is a neighborly member of
$\Kd$.  
\end{prop}
Further, the following result by Kalai specifies the topological space
determined by members of $\Kd$. 

\begin{prop}[Kalai \cite{ka}]\label{prop:kalai}
Let $X\in \Kd$ and let $\beta_1=\beta_1(X;\mathbb{Z}_2)$. If $d\geq 4$, then $X$ triangulates $(S^{\,d-1}\times
S^1)^{\#\beta_1}$ if it is orientable, and it triangulates
$(\TPSSDS)^{\#\beta_1}$ if it is non-orientable. 
\end{prop}

The $4$-manifold constructed by Bagchi-Datta in \cite{bd10}, is tight
neighborly with parameters $(f_0,\beta_1)=(15,3)$ and is a member of
$\Kfour$. Analogous to the class $\Kd$, we define the class $\lKd$ of
triangulated manifolds, where the link of each vertex is a stacked $(d-1)$-ball.
We will use the notation $\Kdn$ and $\lKdn$ to denote neighborly members of
class $\Kd$ and $\lKd$ respectively. From the results in \cite{bd16}, we have the following correspondence:
\begin{prop}[Bagchi \& Datta]\label{prop:bijection}
If $d\geq 4$, then $M\mapsto \partial M$ is a bijection from
$\overline{\cal K}(d+1)$ to $\Kd$. 
\end{prop}
Proposition \ref{prop:bijection} immediately suggests that one can look for
members of $\Kd$ as boundaries of members of $\overline{\cal K}(d+1)$. In
particular, we can
obtain members of $\Kfour$ as boundaries of members of $\lKfive$. This is
exactly the approach we take in this paper. Another easy consequence of the
above correspondence is the following:
\begin{cor}\label{cor:aut}
For $d\geq 4$, let $M,N\in \overline{\cal K}(d+1)$. Then $\varphi:
V(M)\rightarrow V(N)$ is an isomorphism from $M$ to $N$, if and only if it
is an isomorphism from $\partial M$ to $\partial N$. 
\end{cor}
In particular, for $M=N$, we get ${\rm Aut}(M)={\rm Aut}(\partial M)$ for
$M\in \lKd$.

\section{Results}\label{sec:results}
\begin{eg}\label{ex:complexes}{\rm
Let $V=\bigcup_{i=1}^5 \{a_i, b_i, c_i\}$ be a $15$-element set and let $\varphi:V\rightarrow V$ be the permutation defined by $\varphi :=
(a_1,b_1,c_1)(a_2,b_2,c_2)(a_3,b_3,c_3)(a_4,b_4,c_4)(a_5,b_5,c_5)$. For $1\leq
i\leq 12$, let $N_i$ denote the simplicial complex with vertex set $V$ and
the set of facets
\begin{equation}\label{eq:facets} 
\{z_i\}\cup \{u_{i,1},\ldots, u_{i,8}\}\cup \{v_{i,1},\ldots,
v_{i,8}\}\cup \{w_{i,1},\ldots, w_{i,8}\},
\end{equation}
where $z_i = a_1 b_1 c_1 a_2 b_2 c_2$, $v_{i,k} = \varphi(u_{i,k})$ and
$w_{i,k} = \varphi^2(u_{i,k})$ for $1\leq i\leq 12, 1\leq k\leq 8$. The
facets modulo the automorphism $\varphi$ are given in Table
\ref{tbl:facets}.
}
\end{eg}

\begin{table}[t]
{\footnotesize\tt
\begin{center}
\begin{tabular}{|cccc|}
\hline
$ u_{1,1} = a_1 b_1 c_1 a_2 b_2 a_3 $, &
$ u_{1,2} = a_1 b_1 a_2 b_2 a_3 a_4 $, &
$ u_{1,3} = a_1 a_2 b_2 a_3 a_4 a_5 $, &
$ u_{1,4} = a_1 a_2 a_3 a_4 b_4 a_5 $,\\
$ u_{1,5} = a_1 a_3 a_4 b_4 a_5 b_5 $, &
$ u_{1,6} = a_3 b_3 a_4 b_4 a_5 b_5 $, &
$ u_{1,7} = c_2 a_3 b_3 b_4 a_5 b_5 $, &
$ u_{1,8} = b_1 c_2 b_3 b_4 a_5 b_5 $;\\
\hline
$u_{2,1} = a_1 b_1 c_1 a_2 b_2 a_3 $, &
$ u_{2,2} = a_1 b_1 a_2 b_2 a_3 a_4 $, &
$ u_{2,3} = a_1 a_2 b_2 a_3 a_4 a_5 $, &
$ u_{2,4} = a_1 a_2 a_3 a_4 b_4 a_5 $,\\
$ u_{2,5} = a_1 a_2 a_3 b_4 a_5 b_5 $, &
$ u_{2,6} = a_2 a_3 b_3 b_4 a_5 b_5 $, &
$ u_{2,7} = a_3 b_3 b_4 c_4 a_5 b_5 $, &
$ u_{2,8} = b_1 b_3 b_4 c_4 a_5 b_5 $;\\
\hline
$ u_{3,1} = a_1 b_1 c_1 a_2 b_2 a_3 $, &
$ u_{3,2} = a_1 b_1 a_2 b_2 a_3 a_4 $, &
$ u_{3,3} = a_1 b_1 a_2 a_3 a_4 a_5 $, &
$ u_{3,4} = a_1 a_2 a_3 a_4 b_4 a_5 $,\\
$ u_{3,5} = a_1 a_2 a_3 b_4 a_5 b_5 $, &
$ u_{3,6} = a_2 a_3 b_3 b_4 a_5 b_5 $, &
$ u_{3,7} = a_3 b_3 b_4 c_4 a_5 b_5 $, &
$ u_{3,8} = b_2 b_3 b_4 c_4 a_5 b_5 $;\\
\hline
$ u_{4,1} = a_1 b_1 c_1 a_2 b_2 a_3 $, &
$ u_{4,2} = a_1 b_1 a_2 b_2 a_3 a_4 $, &
$ u_{4,3} = a_1 b_1 a_2 a_3 a_4 a_5 $, &
$ u_{4,4} = b_1 a_2 a_3 a_4 a_5 c_5 $,\\
$ u_{4,5} = a_2 a_3 b_3 a_4 a_5 c_5 $, &
$ u_{4,6} = a_2 a_3 b_3 a_4 b_4 a_5 $, &
$ u_{4,7} = a_2 b_3 a_4 b_4 a_5 b_5 $, &
$ u_{4,8} = c_1 b_3 a_4 b_4 a_5 b_5 $;\\
\hline
$ u_{5,1} = a_1 b_1 c_1 a_2 b_2 a_3 $, &
$ u_{5,2} = a_1 b_1 c_1 a_2 a_3 a_4 $, &
$ u_{5,3} = a_1 b_1 a_2 a_3 a_4 a_5 $, &
$ u_{5,4} = a_1 a_2 a_3 a_4 a_5 b_5 $,\\
$ u_{5,5} = a_2 a_3 a_4 b_4 a_5 b_5 $, &
$ u_{5,6} = a_2 a_3 b_3 b_4 a_5 b_5 $, &
$ u_{5,7} = a_2 a_3 b_3 b_4 c_4 b_5 $, &
$ u_{5,8} = a_2 b_3 b_4 c_4 b_5 c_5 $;\\
\hline
$ u_{6,1} = a_1 b_1 c_1 a_2 b_2 a_3 $, &
$ u_{6,2} = a_1 b_1 c_1 a_2 a_3 a_4 $, &
$ u_{6,3} = a_1 b_1 a_2 a_3 a_4 a_5 $, &
$ u_{6,4} = a_1 a_2 a_3 a_4 a_5 b_5 $,\\
$ u_{6,5} = a_2 a_3 a_4 b_4 a_5 b_5 $, &
$ u_{6,6} = a_2 a_3 b_3 a_4 b_4 b_5 $, &
$ u_{6,7} = a_2 a_3 b_3 b_4 b_5 c_5 $, &
$ u_{6,8} = a_2 b_3 b_4 c_4 b_5 c_5 $;\\
\hline
$ u_{7,1} = a_1 b_1 c_1 a_2 b_2 a_3 $, &
$ u_{7,2} = a_1 b_1 c_1 a_2 a_3 a_4 $, &
$ u_{7,3} = a_1 b_1 a_2 a_3 a_4 a_5 $, &
$ u_{7,4} = b_1 a_2 a_3 a_4 a_5 c_5 $,\\
$ u_{7,5} = a_2 a_3 a_4 b_4 a_5 c_5 $, &
$ u_{7,6} = a_2 a_3 b_3 b_4 a_5 c_5 $, &
$ u_{7,7} = a_2 a_3 b_3 b_4 c_4 a_5 $, &
$ u_{7,8} = a_2 b_3 b_4 c_4 a_5 b_5 $;\\
\hline
$ u_{8,1} = a_1 b_1 c_1 a_2 b_2 a_3 $, &
$ u_{8,2} = a_1 b_1 c_1 a_2 a_3 a_4 $, &
$ u_{8,3} = a_1 b_1 a_2 a_3 a_4 a_5 $, &
$ u_{8,4} = b_1 a_2 a_3 a_4 a_5 c_5 $,\\
$ u_{8,5} = a_2 a_3 a_4 b_4 a_5 c_5 $, &
$ u_{8,6} = a_2 a_3 b_3 a_4 b_4 a_5 $, &
$ u_{8,7} = a_2 a_3 b_3 b_4 a_5 b_5 $, &
$ u_{8,8} = a_2 b_3 b_4 c_4 a_5 b_5 $;\\
\hline
$ u_{9,1} = a_1 b_1 c_1 a_2 b_2 a_3 $, &
$ u_{9,2} = a_1 b_1 c_1 a_2 a_3 a_4 $, &
$ u_{9,3} = a_1 b_1 a_2 a_3 a_4 a_5 $, &
$ u_{9,4} = a_1 a_2 a_3 a_4 a_5 b_5 $,\\
$ u_{9,5} = a_2 a_3 a_4 c_4 a_5 b_5 $, &
$ u_{9,6} = a_2 a_3 b_3 a_4 c_4 b_5 $, &
$ u_{9,7} = a_2 a_3 b_3 a_4 b_5 c_5 $, &
$ u_{9,8} = a_2 b_3 a_4 b_4 b_5 c_5 $;\\
\hline
$ u_{10,1} = a_1 b_1 c_1 a_2 b_2 a_3 $, &
$ u_{10,2} = a_1 b_1 c_1 a_2 a_3 a_4 $, &
$ u_{10,3} = a_1 b_1 a_2 a_3 a_4 a_5 $, &
$ u_{10,4} = b_1 a_2 a_3 a_4 a_5 c_5 $,\\
$ u_{10,5} = a_2 a_3 a_4 c_4 a_5 c_5 $, &
$ u_{10,6} = a_2 a_3 b_3 a_4 c_4 a_5 $, &
$ u_{10,7} = a_2 a_3 b_3 a_4 a_5 b_5 $, &
$ u_{10,8} = a_2 b_3 a_4 b_4 a_5 b_5 $;\\
\hline
$ u_{11,1} = a_1 b_1 c_1 a_2 b_2 a_3 $, &
$ u_{11,2} = a_1 b_1 c_1 a_2 a_3 a_4 $, &
$ u_{11,3} = a_1 b_1 a_2 a_3 a_4 a_5 $, &
$ u_{11,4} = a_1 a_2 a_3 a_4 a_5 b_5 $,\\
$ u_{11,5} = a_2 a_3 b_3 a_4 a_5 b_5 $, &
$ u_{11,6} = a_2 b_3 a_4 b_4 a_5 b_5 $, &
$ u_{11,7} = b_2 b_3 a_4 b_4 a_5 b_5 $, &
$ u_{11,8} = b_2 b_3 c_3 a_4 b_4 b_5 $;\\
\hline
$ u_{12,1} = a_1 b_1 c_1 a_2 b_2 a_3 $, &
$ u_{12,2} = a_1 b_1 c_1 a_2 a_3 a_4 $, &
$ u_{12,3} = a_1 b_1 a_2 a_3 a_4 a_5 $, &
$ u_{12,4} = b_1 a_2 a_3 a_4 a_5 c_5 $,\\
$ u_{12,5} = a_2 a_3 b_3 a_4 a_5 c_5 $, &
$ u_{12,6} = a_2 b_3 a_4 b_4 a_5 c_5 $, &
$ u_{12,7} = b_2 b_3 a_4 b_4 a_5 c_5 $, &
$ u_{12,8} = b_2 b_3 c_3 a_4 b_4 a_5 $;\\
\hline
\end{tabular}
\end{center}
}
\caption{\small Facets of complexes modulo automorphism $\varphi$}
\label{tbl:facets}
\end{table}

The following are the main results of this paper.
\begin{theo}\label{thm:thm1}
Let $N\in\lKfiven$ with $f_0(N)=15$ and ${\rm
Aut}(N)\supseteq \mathbb{Z}_3$.  Then
$N\cong N_i$ for some $i\in \{1,2,\ldots,12\}$.
\end{theo}

\begin{theo}\label{thm:thm2}
Let $M\in\Kfourn$ with $f_0(M)=15$ and ${\rm
Aut}(M)\supseteq \mathbb{Z}_3$. Then $M\cong
\partial N_i$ for some $i\in \{1,2,\ldots,12\}$.
\end{theo}

\subsection{Geometric Carrier}
By Theorem \ref{thm:thm1} and Proposition \ref{prop:kalai}, the complexes
$\partial N_i$, $1\leq i\leq 12$,
triangulate $(S^3\times S^1)^{\#3}$ or $(\TPSS)^{\#3}$. Using a
combinatorial topology software such as {\tt simpcomp} \cite{simpcomp}, one
can check that complexes $\partial N_1$ and $\partial N_2$ are non-orientable, while the
complexes $\partial N_i$ for $i\in \{3,4,\ldots,12\}$ are orientable.
Thus, $\partial N_i$ triangulates $(\TPSS)^{\#3}$ for $i=1,2$ and
triangulates $(S^3\times S^1)^{\#3}$ for $3\leq i\leq 12$. We also point
out that the example $N_1$ obtained here is isomorphic to the triangulation
 $N^5_{15}$ obtained by Bagchi and Datta in \cite{bd10}. Consequently, the
triangulation $M^4_{15}$ of $(\TPSS)^{\#3}$ in \cite{bd10} is isomorphic to
$\partial N_1$.

\section{Overview}
In this section, we give a broad outline of the enumeration strategy. By Proposition
\ref{prop:bijection} and Corollary \ref{cor:aut}, to obtain neighborly
members of $\Kfour$, with $\mathbb{Z}_3$ automorphism, we can instead look
for neighborly members of $\lKfive$ with $\mathbb{Z}_3$ automorphism. This
has the advantage that all vertex-links are stacked balls, and by
Proposition \ref{prop:bdns}, we have a succint combinatorial description
for it's dual graph; that it is a tree. Moreover, the dual graph of a
vertex $x$, in a pseudomanifold $X$, is isomorphic to the induced subgraph
$\Lambda(X)[V_x]$, where $V_x$ is the set of facets containing $x$. For 
$M\in \lKd$, let $T_x$ denote the subtree of
$\Lambda(M)$ induced by facets containing $x$. Then from \cite{nsnon}, we
have the following:
\begin{prop}\label{prop:nsnon}
For $M\in \lKdn$, let $T_x$ for $x\in V(M)$ be as defined above. Then
\begin{enumerate}[{\rm (a)}]
\item $\Lambda(M)$ is a two connected graph.
\item $\Lambda(M)$ contains $n(n-d)/(d+1)$ vertices and $n(n-d-1)/d$ edges
where $n=f_0(M)$.
\item $T_x$ contains $n-d$ vertices for each $x\in V(M)$.
\end{enumerate}
\end{prop}

In \cite{nsnon}, a set of facets $S$ of $M\in \lKdn$ was defined to be {\em
critical} in $M$ if each of the connected components of
$\Lambda(M)-S$ contained fewer than $f_0(M)-d$ vertices. The following 
observations were also made there.
\begin{prop}\label{prop:critical}
Let $S$ be a critical set of facets of $M\in \lKdn$. Then the facets in $S$ together
contain all the vertices of $M$.
\end{prop}

\begin{prop}\label{prop:degthree}
Let $M\in \lKdn$ with $f_0(M)>2d+1$. Then the set of facets
with degree three or more in $\Lambda(M)$ together contain all the vertices
of $M$.
\end{prop}

Identifying critical set of facets helps us reduce the possibilities for
members of $\lKfiven$. We already know from Proposition \ref{prop:nsnon},
that the dual graph $\Lambda(M)$ for $M\in \lKdn$ is two connected. When
working with complexes with non-trivial automorphism groups, we can further
narrow down the admissible dual graphs, due to the following observation.

\begin{prop}[cf. {\cite{bdnsinf}}]\label{prop:autdual}
Let $M\in \lKd$, then ${\rm Aut}(M)$ is a subgroup of ${\rm
Aut}(\Lambda(M))$.
\end{prop}

We now summarize what the above propositions imply for $M\in \lKfiven$,
with $f_0(M)=15$ and ${\rm Aut}(M)\supseteq \mathbb{Z}_3$. We have,
\begin{enumerate}[{\rm (a)}]
\item $\Lambda(M)$ is a two connected graph,
\item $\Lambda(M)$ contains $25$ vertices and $27$ edges,
\item $\mathbb{Z}_3$ is a subgroup of ${\rm Aut}(\Lambda(M))$, and
\item for each vertex $x$ of $M$, $x$ appears in $10$ facets of $M$, which
induce a tree on $\Lambda(M)$.
\end{enumerate}

The program we carry out in the next section, leading to the classification
is the following. First we classify all two connected graphs on $25$
vertices, with $27$ edges which exhibit $\mathbb{Z}_3$ symmetry. Then for
each of these graphs, we consider all possible members of
$\lKfiven$ which will have the graph as their dual graph. 

\section{Classification}
\begin{eg}\label{eg:grs}{\rm 
For $r,s\geq 1$, let $G_{r,s}$ be the graph on $3r+3s-2$ vertices, with vertex set 
$V=\{z_0\}\cup (\bigcup_{i=1}^{r+s-1} \{u_i,v_i,w_i\})$ and 
consisting of six edge disjoint paths, namely (see Figure \ref{fig:graphs}\subref{fig:g36}),
}

\noindent\begin{tabular}{lll}
$p_{zu} := z_0u_1\cdots u_r$, & $p_{zv} := z_0v_1\cdots v_r$,
 & $p_{zw} := z_0w_1\cdots w_r$, \\
$p_{uv} := u_ru_{r+1}\cdots u_{r+s-1}v_r$, & 
$p_{vw} := v_rv_{r+1}\cdots v_{r+s-1}w_r$, &
$p_{wu} := w_rw_{r+1}\cdots w_{r+s-1}u_r$.
\end{tabular}
\end{eg}

\begin{eg}\label{eg:trs}
\noindent{\rm For $r,s\geq 1$, let $T_{r,s}$ be the graph on $3r+s$ vertices with vertex set
$V=\{x_1,\ldots,x_s\}\cup (\bigcup_{i=1}^r \{u_i,v_i,w_i\})$ and 
consisting of four edge disjoint paths, namely (see Figure \ref{fig:graphs}\subref{fig:t67}),
}

\noindent\begin{tabular}{llll}
$p_0 := x_1x_1\cdots x_s$, & $p_1 := x_1u_1\cdots u_rx_s$, 
 & $p_2 := x_1v_1\cdots v_rx_s$, &
$p_3 := x_1w_1\cdots w_rx_s$.
\end{tabular}
\end{eg}

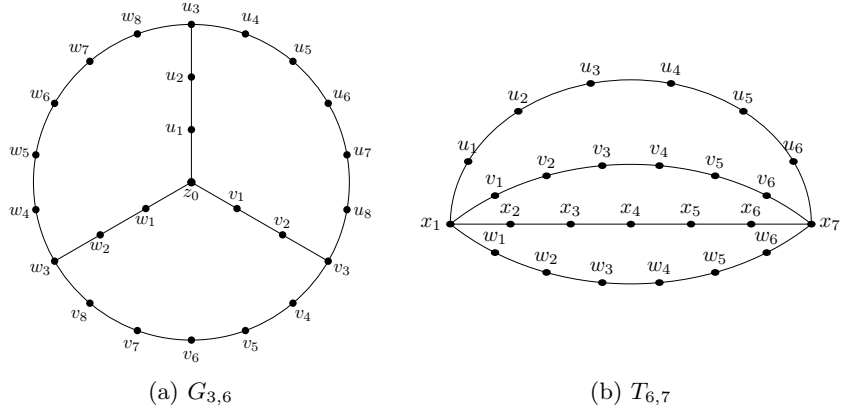
\begin{figure}[htbp]
\centering
\subfloat[$G_{3,6}$\label{fig:g36}] {
\scalebox{0.7} {
\begin{tikzpicture}
\coordinate (center) at (0,0);
\def\radius{3.0cm}
\draw (center) circle[radius=\radius];
\draw[fill] (center) circle[radius=2pt];
\draw (center) node[below] {\small $z_0$};
\draw (center) -- ++(90:\radius);
\draw (center) -- ++(-30:\radius);
\draw (center) -- ++(-150:\radius);

\foreach \i in {1,2} {
\fill (center) ++(90:\i) circle[radius=2pt] 
	node[left] {\small $u_{\i}$};
\fill (center) ++(-30:\i) circle[radius=2pt]
	node[above] {\small $v_{\i}$};
\fill (center) ++(-150:\i) circle[radius=2pt]
	node[below] {\small $w_{\i}$};
}

\foreach \i in {3,4,5,6,7,8} {
\fill (center) ++(150-20*\i:\radius) circle[radius=2pt];
\draw (center) ++(150-20*\i:\radius+0.8em)	
	node {\small $u_{\i}$};
\fill (center) ++(30-20*\i:\radius) circle[radius=2pt];
\draw (center) ++(30-20*\i:\radius+0.8em)
	node {\small $v_{\i}$};
\fill (center) ++(270-20*\i:\radius) circle[radius=2pt];
\draw (center) ++(270-20*\i:\radius+0.8em)
	node {\small $w_{\i}$};

}
\end{tikzpicture}
}
}
\subfloat[$T_{6,7}$\label{fig:t67}]{
\scalebox{0.8}{
\begin{tikzpicture}[yscale=0.8]
\pgfmathsetmacro{\centerA}{3};
\pgfmathsetmacro{\centerB}{0};
\pgfmathsetmacro{\radiusA}{3*sqrt(2)};
\pgfmathsetmacro{\radiusB}{3};
\pgfmathsetmacro{\angleA}{90/7};
\pgfmathsetmacro{\angleB}{180/7};
\pgfmathsetmacro{\initA}{45};
\pgfmathsetmacro{\finA}{135};

\coordinate (cA) at (3,-\centerA);
\coordinate (cB) at (3, 0);
\coordinate (cC) at (3,\centerA);

\draw (6,0) arc (\initA:\finA:\radiusA);
\draw (0,0) arc (180+\initA:180+\finA:\radiusA);
\draw (6,0) arc (0:180:\radiusB);

\draw (0,0) -- (6,0);
\fill (0,0) circle[radius=2pt] node[left] {\small $x_1$};
\fill (6,0) circle[radius=2pt] node[right] {\small $x_7$};

\foreach \i in {2,3,4,5,6} {
\fill (-1+\i,0) circle[radius=2pt] node[above] {\small $x_{\i}$};
}

\foreach \i in {1,2,3,4,5,6} {
\fill (cA) ++(\finA-\angleA*\i:\radiusA) 
	circle[radius=2pt] node[above] {\small $v_{\i}$};
\fill (cB) ++(180-\angleB*\i:\radiusB)
	circle[radius=2pt] node[above] {\small $u_{\i}$};
\fill (cC) ++(180+\initA+\angleA*\i:\radiusA)
	circle[radius=2pt] node[above] {\small $w_{\i}$};

}
\end{tikzpicture}
}
}
\caption{\small Graphs $G_{3,6}$ and $T_{6,7}$}
\label{fig:graphs}
\end{figure}

\begin{lemma}\label{lem:graphs}
Let $G$ be a $2$-connected graph on $25$ vertices with $27$ edges. If
${\rm Aut}(G)\supseteq \mathbb{Z}_3$, then $G\cong G_{r,s}$ for some
$r,s>0$ with $r+s=9$ or $G\cong T_{r,s}$ for some $r,s>0$ with $3r+s=25$. 
\end{lemma}
We will defer the proof of Lemma \ref{lem:graphs} to the Appendix A. For
$M\in \lKfiven$ with $f_0(M)=15$, by Proposition \ref{prop:degthree}, we see that the set of facets
with degree three or more in $\Lambda(M)$ together must contain all the
vertices of $M$. But the graphs $T_{r,s}$ contain only two vertices of
degree three or more, and hence can contain at most $2\times 6=12$ vertices,
which cannot cover all vertices of $M$, as $f_0(M)=15$. Thus, we have a
further constraint on the graph $\Lambda(M)$. In particular from
Proposition \ref{prop:autdual} and Lemma \ref{lem:graphs} we have,
\begin{lemma}\label{lem:dualgraph}
Let $M\in \lKfiven$ with $f_0(M)=15$ and ${\rm Aut}(M)\supseteq
\mathbb{Z}_3$. Then $\Lambda(M)\cong G_{r,s}$ for some $r,s>0$ with $r+s=9$.
\end{lemma}

The following result from \cite{nsnon} further restricts the structure of
$\Lambda(M)$, for $M\in \lKdn$. 

\begin{prop}\label{prop:pathprop}
Let $M\in \lKdn$ with $f_0(M) > 2d+1$. Let $u_0u_1\cdots u_r$ be a path in
$\Lambda(M)$ where all the internal vertices $u_i$, $1\leq i\leq r-1$, have
degree two in $\Lambda(M)$. Let $x_i$ be the unique element in
$u_{i-1}\backslash u_i$ for $1\leq i\leq r$. Then we have,
\begin{enumerate}[{\rm (a)}]
\item $x_1,\ldots, x_r$ are distinct,
\item $x_i\in u_0$ for $1\leq i\leq r$,
\item $r\leq d+1$.
\end{enumerate}
\end{prop}
By Part (c) of the above proposition, $\Lambda(M)$ contains induced paths of
length at most $d+1$. For $M\in \lKfiven$, therefore, $\Lambda(M)$ contains
induced paths of length at most $6$. By Lemma \ref{lem:dualgraph}, we know
that $\Lambda(M)\cong G_{r,s}$ for some $r+s=9$. Since $G_{r,s}$ contains
the induced path of length at least $\max(r,s)$, we must that
$\max(r,s)\leq 6$. Together with the constraint $r+s=9$, we conclude
\begin{lemma}\label{lem:graphdual2}
Let $M\in \lKfiven$ with $f_0(M)=15$. If ${\rm Aut}(M)\supseteq
\mathbb{Z}_3$,  then $\Lambda(M)\cong G_{r,9-r}$ for some $r\in \{3,4,5,6\}$.
\end{lemma}
 

Let $M\in \lKd$ and for each $x\in V(M)$, let $T_x$ denote the subtree of
$\Lambda(M)$ induced by the facets of $M$ containing $x$. We note the
following:
\begin{lemma}\label{lem:treelemma}
For $M\in \lKd$, let $T_x$ for $x\in V(M)$ be as defined. Then,
\begin{enumerate}[{\rm (a)}]
\item $T_x\neq T_y$ for $x\neq y$,
\item If $\sigma\in V(\Lambda(M))$ is a leaf of some tree $T_x$, then
$d_{\Lambda(M)}(\sigma) < 3$.
\end{enumerate}
\end{lemma}
\begin{proof}
We first prove (a). Suppose $T_x=T_y$ for some $x\neq y$. Since $T_x\neq
\Lambda(M)$, and $\Lambda(M)$ is connected graph, there exists an edge
$uv$ in $\Lambda(M)$ such that $u\in V(T_x)$ and $v\neq V(T_x)$. Now since
$u\in V(T_x)=V(T_y)$, we have $\{x,y\}\subseteq u$. Since $v\neq T_x,T_y$, we have
$\{x,y\}\cap v=\emptyset$. Thus $\{x,y\}\subseteq u\backslash v$, which is a
contradiction as $uv$ is an edge in $\Lambda(M)$. 

To prove (b), suppose $\sigma$ is a leaf of tree $T_x$ and
$d_{\Lambda(M)}(\sigma) \geq 3$. Since at most one neighbor of $\sigma$ is
on $T_x$, we have at least two neighbors of $\sigma$, say $\alpha$ and
$\beta$ which are not on $T_x$. Thus $x\not\in \alpha,\beta$. But then
$\sigma\backslash \{x\}\subseteq \alpha,\beta$. This is a contradiction as
$\sigma\backslash \{x\}$ is a face of co-dimension one which is contained
in three facets, namely $\sigma,\alpha$ and $\beta$. This completes the
proof.
\end{proof}

Let $M\in \lKd$ and let $\{T_x: x\in V(M)\}$ be the collection of trees as
before. Clearly the trees $T_x$ and $T_y$ intersect if and only if $x\in
\st{M}{y}$. Thus the number of trees that a tree $T_x$ intersects (counting
itself), is same as the number of vertices in it's star. From Proposition
\ref{prop:bdns} we have (since $\st{M}{x}$ is a stacked $d$-ball),
\[  f_0(\st{M}{x}) = f_d(\st{M}{x}) + d.  \]
However, $f_d(\st{M}{x})$ is the number of vertices in the tree $T_x$, and
hence we can write the above as,
\begin{equation}
\text{Number of trees intersected by } T_x = |V(T_x)| + d.
\end{equation}
We can however, count the number of trees intersected by $T_x$ in the
following way. Designate a fixed vertex $r\in V(T_x)$ to be the {\em root}.
Next we orient each edge $uv$ of $T_x$ as $\overrightarrow{uv}$ where $v$
is the vertex nearer to $r$ (see Figure \ref{fig:oriented}). To each oriented edge $\overrightarrow{uv}$ of
$T_x$, we associate a label $l(\overrightarrow{uv})=y$ where $y$ is
the unique element of $u\backslash v$. Now there are $d+1$ trees that
intersect $T_x$ (including itself), at vertex $r$. For any tree $T_y$ which
intersects $T_x$, but does not contain $r$, there must be an oriented edge
$e$ in $T_x$ such that $l(e)=y$. Thus the number of trees that $T_x$
intersects is at most $(d+1) + |V(T_x)|-1 = |V(T_x)|+d$. For the equality
to hold, which it should for $M\in \lKd$, all the edge labels must be
distinct, and different from vertices of the facet $r$. We note this
observation as the following lemma.

\begin{figure}[htbp]
\centering
\begin{tikzpicture}[scale=1.5]
\coordinate (R) at (0,0);
\coordinate (A) at (-0.1,1);
\coordinate (B) at (-1.1,1.5);
\coordinate (C) at (-0.7,-0.5);
\coordinate (D) at (-1.4,-1);
\coordinate (E) at (-2.2,-1);
\coordinate (F) at (1.2,0.3);
\coordinate (G) at (2.2,1.1);
\coordinate (H) at (2.3,0.2);
\coordinate (I) at (2.4,-0.8);

\fill (R) circle[radius=1pt] node {};
\fill (A) circle[radius=1pt] node {};
\fill (B) circle[radius=1pt] node {};
\fill (C) circle[radius=1pt] node {};
\fill (D) circle[radius=1pt] node {};
\fill (E) circle[radius=1pt] node {};
\fill (F) circle[radius=1pt] node {};
\fill (G) circle[radius=1pt] node {};
\fill (H) circle[radius=1pt] node {};
\fill (I) circle[radius=1pt] node {};

\draw[reverse directed] (R) -- (A);
\draw[reverse directed] (A) -- (B);
\draw[reverse directed] (R) -- (C);
\draw[reverse directed] (C) -- (D);
\draw[reverse directed] (D) -- (E);
\draw[reverse directed] (R) -- (F);
\draw[reverse directed] (F) -- (G);
\draw[reverse directed]	(F) -- (H);
\draw[reverse directed] (H) -- (I);

\draw (R) node[below right] {$r$};
\end{tikzpicture}
\caption{\small Oriented tree with root $r$}
\label{fig:oriented}
\end{figure}
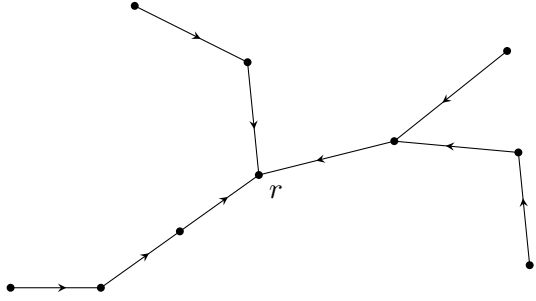

\begin{lemma}\label{lem:oriented}
For $M\in \lKd$ and $x\in V(M)$, let $T_x$ be an oriented tree with root
$r$, as described
before. Then the labels on the oriented edges are distinct, and different
from vertices of the facet $r$.
\end{lemma}

\begin{lemma}\label{lem:pathlemma2}
Let $M\in \lKd$ and let $u_0u_1\cdots u_r$ be a path in $\Lambda(M)$ with
$r<d+1$. Let $x_i$ be the unique element in $u_{i-1}\backslash u_i$, and
$y_i$ be the unique element in $u_i\backslash u_{i-1}$ for  $1\leq i\leq
r$. Then we have,
\begin{enumerate}[{\rm (a)}]
\item $\{x_1,x_2,\ldots,x_r\}\subseteq u_0\backslash u_r$, and
\item $\{y_1,y_2,\ldots,y_r\}\subseteq u_r\backslash u_0$.
\end{enumerate}
\end{lemma}
\begin{proof}
Since $r<d+1$, there exists $z\in M$ such that
$\{u_0,u_1,\ldots,u_r\}\subseteq V(T_z)$. Part (a) follows by orienting
$T_z$ with the $u_r$ as root and applying Lemma \ref{lem:oriented},  and
Part (b) follows similarly by orienting $T_z$ with
$u_0$ as root.
\end{proof}

Let $M\in \lKfiven$ with $f_0(M)=15$ and ${\rm Aut}(M)\supseteq \mathbb{Z}_3$.  Recall that
$\Lambda(M)\cong G_{r,9-r}$ for some
$3\leq r\leq 6$ (see Example \ref{eg:grs}). We identify the facets of
$M$ with vertices of $G_{r,9-r}$. Let $\varphi$ be an automorphism of $M$. Then $\varphi$ induces an
automorphism $\bar{\varphi}: u\mapsto \varphi(u)$ of $\Lambda(M)$.
Following lemma has been proved in \cite{bdnsinf}.
\begin{lemma}\label{lem:authomo}
Let $M\in \lKd$. For $\varphi\in {\rm Aut}(M)$, let $\bar{\varphi}$ be an
induced automorphism of $\Lambda(M)$. Then $\varphi\mapsto \bar{\varphi}$ is
an injective homomorphism from ${\rm Aut}(M)$ to ${\rm Aut}(\Lambda(M))$.
\end{lemma}  
Now let $\varphi$ be an order three automorphism of $M$. Then, it induces
an  order three automorphism $\bar{\varphi}$ of $\Lambda(M)$. Since
$\Lambda(M)\cong G_{r,9-r}$, $\bar{\varphi}$ is an order three automorphism
of $G_{r,9-r}$. Without loss of generality, we may assume $\varphi$ is such
that, $\bar{\varphi} = \prod_{i=1}^{8}(u_i,v_i,w_i)$. We show that
$\varphi$ does not have any fixed points.
\begin{lemma}\label{lem:fixedpoints}
Let $M\in \lKfiven$ with $f_0(M)=15$ and ${\rm Aut}(M)\supseteq \mathbb{Z}_3$. If $\varphi$
is an order three automorphism of $M$, then $\varphi$ has no fixed points.
\end{lemma}
\begin{proof}
By Lemma \ref{lem:graphdual2}, $\Lambda(M)\cong G_{r,9-r}$ for some $3\leq r\leq 6$. We
identify facets of $M$ with vertices of $G_{r,9-r}$. Without loss of generality assume $\varphi$ induces
the automorphism $\bar{\varphi}=\prod_{i=1}^8 (u_i,v_i,w_i)$ of
$G_{r,9-r}$. Now the orbits of $\varphi$ are either singleton or contain $3$
elements. Let $x$ be a fixed point of $\varphi$. Let $V_x$ be the facets of
$M$ containing $x$. By definition of $\bar{\varphi}$, we have
$\bar{\varphi}(V_x)=V_x$. Thus $V_x$ is union of orbits of
$\bar{\varphi}=\prod_{i=1}^8 (u_i,v_i,w_i)$. Since $|V_x|=10$ and each
orbit of $\bar{\varphi}$ has cardinality three or one, $V_x$ must contain
a fixed point of $\bar{\varphi}$. Since $z_0$ is the only fixed point of
$\bar{\varphi}$, we conclude $z_0\in V_x$. Also, observe that we must have,
\[ |V_x\cap \{u_1,\ldots,u_8\}| = |V_x\cap \{v_1,\ldots,v_8\}|
   = |V_x\cap \{w_1,\ldots, w_8\}|. \]
Thus we must have $|V_x\cap \{u_1,\ldots,u_8\}|=3$. Since $V_x$ induces a
connected subgraph of $G_{r,9-r}$ and $r\geq 3$, we have $V_x\cap
\{u_1,\ldots, u_8\}=\{u_1,u_2,u_3\}$. Similarly, $V_x\cap \{v_1,\ldots,
v_8\}=\{v_1,v_2,v_3\}$ and $V_x\cap \{w_1,\ldots, w_8\}=\{w_1,w_2,w_3\}$.
Thus $V_x = \{z_0\}\cup (\cup_{i=1}^3 \{u_i,v_i,w_i\})$. However if $\varphi$
has one fixed point, it has at least three fixed points. Let $y\neq x$ be
another fixed point of $\varphi$. Then we will have $V_y=\{z_0\}\cup
(\cup_{i=1}^3 \{u_i,v_i,w_i\})=V_x$, and thus $T_x=T_y$, which contradicts
 Lemma \ref{lem:treelemma}. This proves the lemma.
\end{proof}

\begin{defn}[{Class ${\cal C}$ of complexes}]\label{defn:calC}{\rm
Let $V := \bigcup_{i=1}^5 \{a_i,b_i,c_i\}$ be a $15$-vertex set. Let $V$ be
ordered as $a_1<b_1<c_1<a_2<b_2<\cdots < b_5<c_5$. Let $\Phi :=
\prod_{i=1}^5 (a_i,b_i,c_i)$. We will denote the orbit $\{a_i,b_i,c_i\}$ of
$\Phi$ as $\Phi_i$. Let us define the class ${\cal C}$ of
complexes as
\begin{equation}\label{eq:complexes}
{\cal C} := \{M\in \lKfiven: V(M)=V, \Phi\in {\rm Aut}(M)\}.
\end{equation}
}
\end{defn}
We notice that
any $M\in \lKfiven$ with $f_0(M)=15$ and ${\rm Aut}(M)\supseteq
\mathbb{Z}_3$ is isomorphic to a member of the collection ${\cal C}$.
Therefore it suffices to enumerate ${\cal C}$ up to isomorphism. Before we
proceed with enumeration, we consider an efficient string representation of
a member of the class ${\cal C}$.

If $M\in {\cal C}$, then we know that $\Lambda(M)\cong G_{r,9-r}$ for some
$3\leq r\leq 6$. We identify the facets of $M$ with vertices of
$G_{r,9-r}$. We may assume that $\Phi$ induces the automorphism
$\prod_{i=1}^8 (u_i,v_i,w_i)$ of $G_{r,9-r}$. To each facet $u$ of $M$, we
associate a string $\bkt{u}=a_1a_2\ldots a_k$ where $a_1<a_2<\cdots <a_k$
are vertices of $u$. With the complex $M$, we associate the string
representation, which we denote by ${\rm str}(M)$ as,
\begin{equation}
{\rm str}(M) := \bkt{z_0}+\bkt{u_1}+\bkt{u_2}+\cdots+\bkt{u_8},
\end{equation}
where $+$ denotes the concatenation of strings. Note that the above
representation uniquely specifies a complex in ${\cal C}$ as remaining
facets may be obtained by applying the automorphism $\Phi$. Clearly any other
orbit such as $\{z_0,v_1,\ldots, v_8\}$ or $\{z_0,w_1,\ldots,w_8\}$ could have
been used. However, we assume that $\{z_0,u_1,\ldots, u_8\}$ is the one
that yeilds lexicographically least representation, i.e, we assume
$\bkt{u_1} < \min(\bkt{v_1},\bkt{w_1})$.

\subsection{Lexicographic Enumeration}
Throughout the remainder of the paper, let  $V,{\cal C}$ and $\Phi$ be as in Definition \ref{defn:calC}. We order ${\cal C}$ with
the ordering $M_1\leq M_2$ for $M_1,M_2\in {\cal C}$ if ${\rm str}(M_1)\leq
{\rm str}(M_2)$, where the ordering on strings is lexicographic. We say
$M\in {\cal C}$ is {\em minimal}, if $M\leq N$ for all $N\in {\cal C}$, such that $N$
is isomorphic to $M$. Thus each isomorphism class contains exactly one
minimal element. We will look for such minimal members of ${\cal C}$.

\begin{lemma}\label{lem:perms}
Let $M\in {\cal C}$. If $\Gamma$ is a permutation of $V$ such
that $\Phi\Gamma = \Gamma\Phi^i$ for some $i\in \{1,2\}$, then
$\Gamma(M)\in {\cal C}$.
\end{lemma}
\begin{proof}
Clearly $\Gamma(M)\in \lKfiven$ for $M\in \lKfiven$.
We need to show that $\Phi$ is an automorphism of $\Gamma(M)$. Suppose
$\Phi\Gamma=\Gamma\Phi^i$, where $i\in \{1,2\}$. Let $u$ be a facet of
$\Gamma(M)$. Then $u=\Gamma(v)$ for
some facet $v$ of $M$. Now $\Phi(u)=\Phi\Gamma(v)=\Gamma\Phi^i(v)=\Gamma(w)$ where
$w=\Phi^i(v)$ is a facet of $M$, and hence $\Gamma(w)$ is a facet of $\Gamma(M)$.
This $\Phi$ maps facets to facets in $\Gamma(M)$, and hence is an automorphism
of $\Gamma(M)$. The lemma follows.
\end{proof}

Lemma \ref{lem:perms} implies that for a mimimal complex $M\in {\cal C}$
and for $\Gamma$, a permutation of $V$ satisfying the conditions in the
lemma, we must have $M\leq \Gamma(M)$. Next we define some permutations of
$V$, which satisfy the conditions in Lemma \ref{lem:perms}.
\begin{itemize}
\item $\pi_i := (a_i,b_i,c_i)$ for $1\leq i\leq 5$,
\item $\pi_{i,j} := (a_i,a_j)(b_i,b_j)(c_i,c_j)$ for $1\leq i<j\leq 5$,
\item $\gamma_{\alpha,\beta} :=
(\alpha_1,\beta_1)(\alpha_2,\beta_2)(\alpha_3,\beta_3)(\alpha_4,\beta_4)(\alpha_5,\beta_5)$,
where $\{\alpha_i,\beta_i\}\subseteq \{a_i,b_i,c_i\}$ for $1\leq i\leq 5$.
\end{itemize}
We can think of permutations $\pi_i$ as shifting the elements cyclically within
an orbit. The permuatations $\pi_{i,j}$ ``interchange'' the orbits $i$ and
$j$, while the permutations $\gamma_{\alpha,\beta}$ interchange an adjacent pair in
each of the orbits. We will need the above permutations in pruning our
candidates for minimal element of ${\cal C}$.

\begin{lemma}\label{lem:zmininal}
Let $M\in {\cal C}$ be minimal. Then $\bkt{z_0}=a_1b_1c_1a_2b_2c_2$.
\end{lemma}
\begin{proof}
Since $\Phi(z_0)=z_0$, $z_0$ must be a union of $\Phi$-orbits. Since
$|z_0|=6$, it contains exactly two orbits. For $M$ to be minimal,
$\bkt{z_0}$ should be minimal for $M$, among all complexes in it's
isomorphism class. Thus $z_0$ should be union of orbits $\Phi_1$ and
$\Phi_2$, for otherwise using one of the permutations $\pi_{i,j}$ we can get
the complex $\pi_{i,j}(M)$ with lexicographically smaller $z_0$. Thus
$z_0=\Phi_1\cup \Phi_2$, or $\bkt{z_0}=a_1b_1c_1a_2b_2c_2$.
\end{proof}

We introduce a succint representation of complexes in ${\cal C}$. Recall
that with an oriented edge $\ovr{uv}$ in $\Lambda(M)$, we associated a
label $l(\ovr{uv})$ as the unique element of $u\backslash v$. To a
complex $M\in {\cal C}$, with $\Lambda(M)\cong G_{r,9-r}$, we associate tuples $X=(x_1,x_2,\ldots, x_9)$
and $Y=(y_1,y_2,\ldots,y_9)$, where (see Figure \ref{fig:complex}) 
\begin{equation}\label{eq:tuples}
x_i = 
\begin{cases}
l(\ovr{z_0u_1}) & \text{ if }  i=1, \\
l(\ovr{u_{i-1}u_i}) & \text{ if } 2\leq i\leq 8, \\
l(\ovr{u_8v_r}) & \text{ if } i=9,
\end{cases}
\quad \text{and}
\quad y_i =
\begin{cases}
l(\ovr{u_1z_0}) & \text{ if } i=1, \\
l(\ovr{u_iu_{i-1}}) & \text{ if } 2\leq i\leq 8, \\
l(\ovr{v_ru_8}) & \text{ if } i=9.
\end{cases}
\end{equation}
Clearly, the triple $(z_0,X,Y)$ uniquely specifies a complex in ${\cal
C}$. Further, by Lemma \ref{lem:zmininal}, for minimal complexes, $z_0$ is
constant, and hence the pair $(X,Y)$ uniquely specifies a minimal complex in
${\cal C}$. We will frequently make use of the following lemma. In the
remainder of the paper, we shall always assume that for a minimal complex
$M\in {\cal C}$, $X=(x_1,\ldots,x_9)$ and $Y=(y_1,\ldots, y_9)$ are the
tuples as defined in (\ref{eq:tuples}).

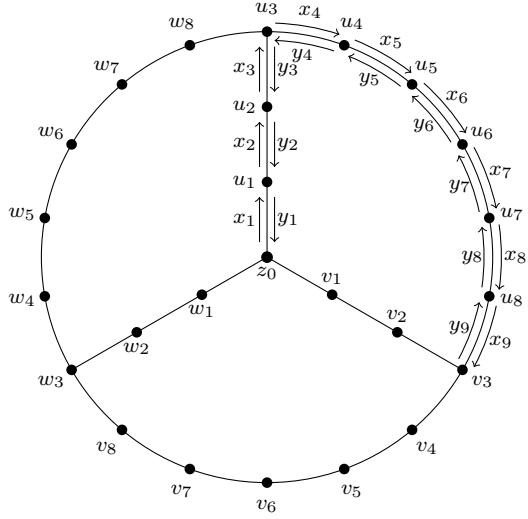
\begin{figure}[htbp]
\centering
\begin{tikzpicture}[scale=1]
\coordinate (center) at (0,0);
\def\radius{3.0cm}
\draw (center) circle[radius=\radius];
\draw[fill] (center) circle[radius=2pt];
\draw (center) node[below] {\scriptsize $z_0$};
\draw (center) -- ++(90:\radius);
\draw (center) -- ++(-30:\radius);
\draw (center) -- ++(-150:\radius);

\foreach \i in {1,2} {
\fill (center) ++(90:\i) circle[radius=2pt] 
	node[left] {\scriptsize $u_{\i}$};
\fill (center) ++(-30:\i) circle[radius=2pt]
	node[above] {\scriptsize $v_{\i}$};
\fill (center) ++(-150:\i) circle[radius=2pt]
	node[below] {\scriptsize $w_{\i}$};
}

\foreach \i in {3,4,5,6,7,8} {
\fill (center) ++(150-20*\i:\radius) circle[radius=2pt];
\draw (center) ++(150-20*\i:\radius+0.8em)	
	node {\scriptsize $u_{\i}$};
\fill (center) ++(30-20*\i:\radius) circle[radius=2pt];
\draw (center) ++(30-20*\i:\radius+0.8em)
	node {\scriptsize $v_{\i}$};
\fill (center) ++(270-20*\i:\radius) circle[radius=2pt];
\draw (center) ++(270-20*\i:\radius+0.8em)
	node {\scriptsize $w_{\i}$};

}

\foreach \i in {1,2,3} {
	\draw[->] (-0.1,\i-0.8) -- (-0.1,\i-0.2);
	\draw (0,\i-0.5) node[left] {\scriptsize $x_{\i}$};
	\draw[->] (0.1,\i-0.2) -- (0.1,\i-0.8);
	\draw (0,\i-0.5) node[right] {\scriptsize $y_{\i}$};
}

\foreach \i in {4,5,6,7,8,9} {
	\draw[->] (center) ++(168-20*\i:\radius+0.3em) arc 
				(168-20*\i:152-20*\i:\radius+0.3em);
	\draw[<-] (center) ++(168-20*\i:\radius-0.3em) arc
				(168-20*\i:152-20*\i:\radius-0.3em);
}

\foreach \i in {4,5,6,7,8,9} {
	\draw (center) ++(160-20*\i:\radius+0.8em) node {\scriptsize $x_{\i}$};
	\draw (center) ++(160-20*\i:\radius-0.7em) node {\scriptsize $y_{\i}$};
}

\end{tikzpicture}
\caption{\small Succint representation of a minimal complex in ${\cal C}$}
\label{fig:complex}
\end{figure}

\begin{lemma}\label{lem:leaveseq}
Let $M\in {\cal C}$ be minimal. For $i\in \{3,4,5\}$, define $m_i=\min\{j:
y_j \in \Phi_i\}$. For $i\in \{1,2\}$, define $n_i=\min\{j:x_j\in
\Phi_i\}$. Then we have,
\begin{enumerate}[{\rm (a)}]
\item $m_3 < m_4 < m_5$.
\item $n_2 < n_1$.
\item $y_{m_i} = a_i$ for $i\in \{3,4,5\}$.
\item $x_{n_i} = c_i$  for $i\in \{1,2\}$.
\end{enumerate}
\end{lemma}
\begin{proof}
To prove part (a), we prove $m_3<m_4$ and $m_4<m_5$. Assume that 
$m_3 > m_4$. Let $\pi = \pi_{3,4}$. Consider the complex $M'=\pi(M)$. Since
$\pi(z_0),\pi(u_1),\ldots,\pi(u_8)$ is an orbit of $M'$ under $\Phi$, we have
\[ {\rm str}(M') \leq \bkt{\pi(z_0)}+\bkt{\pi(u_1)}+\cdots+\bkt{\pi(u_8)} .\]
Note that for $i < m_4 < m_3$, we have $\pi(u_i)=u_i$ as none of the
$u_i$'s contain elements from $\Phi_3$ or $\Phi_4$. However, for $i=m_4$,
we have $u_i = S\cup \{y\}$ where $S\cap \Phi_3=S\cap \Phi_4=\emptyset$ and
$y\in \Phi_4$. Thus $\pi(u_i)=\pi(S)\cup \{\pi(y)\} = S\cup
\{\pi(y)\}$. However, $\pi(y)\in \Phi_3$ as $y\in \Phi_4$, and hence
$\pi(y) < y$, or $\bkt{\pi(u_i)} < \bkt{u_i}$. Thus 
$\pi(M)<M$, a contradiction to minimality of $M$. This proves $m_3 < m_4$.
Similarly we can show that $m_4 < m_5$. 
\smallskip

To prove part (b), assume that $n_1 < n_2$. Let $\pi =
\pi_{1,2}$. For $i<n_1<n_2$, $\Phi_1\subseteq u_i$ and $\Phi_2\subseteq
u_i$, and hence $\pi(u_i)=u_i$. For $i=n_1$, we have $u_i=S_1\cup
S_2\cup S_3$, where $S_1=\Phi_1\backslash\{x_i\}$, $S_2=\Phi_2$ and $S_3\subseteq
\Phi_3\cup \Phi_4\cup \Phi_5$. Thus, $\pi(u_i)=\pi(S_1)\cup \pi(S_2)\cup
\pi(S_3)$. But $\pi(S_2)=\Phi_1$ as $S_2=\Phi_2$. Thus $[\pi(u_i)]$
contains first three positions from $\Phi_1$, whereas $[u_i]$ contains
first two positions from $\Phi_1$. Hence $\pi(u_i) < u_i$ for $i=n_1$, and
hence $\pi(M)<M$, a contradiction to minimality of $M$. This proves $n_2 <
n_1$.
\smallskip

Parts (c) and (d) may be proved using the permuations
$\pi_i=(a_i,b_i,c_i)$. Informally, in the string representation of $M$,
when an orbit element appears for the first time, we can always permute the
orbit, so that it is the least element $a_i$, for $i^{th}$ orbit.
Similarly, when an orbit element leaves for the first time, we can always
permute the elements so that it is the greatest element $c_i$ for $i^{th}$
orbit. 
\end{proof}

For $i\in \{1,2,\ldots,12\}$, let $N_i$ be the simplicial complexes as defined in Example
\ref{ex:complexes}.

\begin{lemma}\label{lem:dg36}
Let $M\in {\cal C}$ be minimal with $\Lambda(M)\cong G_{3,6}$. Then $M\cong
N_1$.
\end{lemma}
\begin{proof}
Let $X=(x_1,x_2,\ldots,x_9)$ and $Y=(y_1,y_2,\ldots, y_9)$ be the tuple
associated with $M$. We claim the following:
\begin{enumerate}[{\rm (a)}]
\item $(y_1,y_2,y_3)=(a_3,a_4,a_5)$.
\item $(x_1,x_2)=(c_2,c_1)$.
\item $x_3\in \{b_1,b_2\}$.
\end{enumerate}
Recall that a set $S$ of facets of $M$ is
critical in $M$ if each component of $\Lambda(M)-S$ has less than $10$
vertices. Thus the set of facets $u_3,v_3,w_3$ is critical. By Proposition
\ref{prop:critical}, $V=u_3\cup v_3\cup w_3$. Since $v_3=\Phi(u_3)$ and
$w_3=\Phi^2(u_3)$, we conclude that $u_3$ intersects each $\Phi$-orbit.
Since $z_0=\Phi_1\cup \Phi_2$, and $u_3\backslash z_0=\{y_1,y_2,y_3\}$, we conclude that $y_1,y_2$ and $y_3$ each
come from distinct orbits among $\Phi_3,\Phi_4$ and $\Phi_5$. By Parts (a)
and (c) of Lemma
\ref{lem:leaveseq}, we must have $(y_1,y_2,y_3)=(a_3,a_4,a_5)$. This proves
(a). 

By Lemma
\ref{lem:pathlemma2}, $\{x_1,x_2,x_3\}\subseteq z_0\subseteq \Phi_1\cup
\Phi_2$. Thus, by Parts (b) and (d) of Lemma \ref{lem:leaveseq}, we
have $x_1\in \Phi_2$ and further that $x_1=c_2$. We claim that
$x_2=l(\ovr{u_1u_2})\in \Phi_1$. If possible, let $l(\ovr{u_1u_2})\in
\Phi_2$. Then $x_3=l(\ovr{u_2u_3})\in \Phi_1$, otherwise $u_3$ will
not intersect $\Phi_2$. Let us count the vertices in $T_{x_3}$. Clearly
$T_{x_3}$ contains $z_0,u_1,u_2$. Further notice that none of the edges
oriented away from $z_0$ on the paths $z_0\cdots v_3$ and $z_0\cdots w_3$ have the
label $x_3$. Thus, $\{v_1,v_2,v_3,w_1,w_2,w_3\}\subseteq V(T_{x_3})$. 
We have already accounted for $6+3=9$ vertices of $T_{x_3}$. By Part (b) of
Lemma \ref{lem:treelemma} 
$v_3,w_3$ cannot be leaves of $T_{x_3}$, therefore we must have at least two more
vertices in $T_{x_3}$. Thus, $T_{x_3}$ contains at least $11$ vertices, a
contradiction. Therefore, $x_2\in \Phi_1$, and hence by Part (d) of Lemma
\ref{lem:leaveseq}, $x_2=c_1$. This proves (b).

We first show that if $x_3\in \Phi_1$, then $x_3=b_1$ and if $x_3\in
\Phi_2$, then $x_3=b_2$. Assume that $x_3\in \Phi_1$ and $x_3\neq
b_1$. Then $x_3=a_1$. Let $\pi :=
(a_1,c_1)(a_2,c_2)(a_3,b_3)(a_4,b_4)(a_5,b_5)$. We will show that
$\pi(M)<M$. Let $M^{\prime}=\pi(M)$. For clarity, we will denote the
facet of $M$ corresponding to vertex $u$ of the dual graph $G_{3,6}$ as
$M(u)$ and similarly facet of $M^{\prime}$ corresponding to vertex $u$ of $G_{3,6}$
as $M^{\prime}(u)$. We notice that $M^{\prime}(z_0)=M(z_0)=a_1b_1c_1a_2b_2c_2$. Since
$M^{\prime}(u_1)$ is the lexicographically least neighbor of $M^{\prime}(z_0)$, we see that
$M^{\prime}(u_1)=\pi(M(v_1))$ (Lexicographically least neighbor is the one, along
which $c_2$ leaves). It follows that $M^{\prime}(u_i)=\pi(M(v_i))$ for
$1\leq i\leq 3$. From parts
(a) and (b) we have,
\begin{align*}
{\rm str}(M) &=
a_1b_1c_1a_2b_2c_2+a_1b_1c_1a_2b_2a_3+a_1b_1a_2b_2a_3a_4+b_1a_2b_2a_3a_4a_5+\cdots \\
{\rm str}(M^{\prime}) &=
[M^{\prime}(z_0)]+[M^{\prime}(u_1)]+[M^{\prime}(u_2)]+[M^{\prime}(u_3)]+\cdots \\
			  &=
[\pi(M(z_0))]+[\pi(M(v_1))]+[\pi(M(v_2))]+[\pi(M(v_3))]+\cdots \\
			  &=
a_1b_1c_1a_2b_2c_2+a_1b_1c_1a_2b_2a_3+a_1b_1a_2b_2a_3a_4+a_1a_2b_2a_3a_4a_5+\cdots
\\
&< {\rm str}(M).
\end{align*}
This contradicts the minimality of $M$. Hence $x_3\in \Phi_1$ implies
$x_3=b_1$. Similarly it can be shown that $x_3\in \Phi_2$ implies
$x_3=b_2$. This proves (c).

Let us deduce more about the arrangement of the trees $T_x$ for $x\in V$.
Because of the automorphism $\Phi$, we only need to know the trees of
$a_1,a_2,a_3,a_4,a_5$. We note that for $x\in\{a_3,a_4,a_5\}$, $x\not\in
z_0\cup v_3\cup w_3$, and hence $T_x\subseteq \Lambda(M)-\{z_0,v_3,w_3\}$. Thus the trees $T_{a_3},T_{a_4},T_{a_5}$ are contained
in the union of the paths $w_4w_5\cdots u_3\cdots u_8$ and
$u_3\cdots u_1$. Since $y_3=a_5$ leaves along the edge $\ovr{u_3u_2}$, we
conclude that $T_{a_5}$ is contained in the arc $w_4\cdots u_3\cdots u_8$. 
As $T_{a_5}$ has $10$ vertices, we have following cases:
\begin{figure}[htbp]
\centering
\subfloat[Case 1:\label{fig:c361}] {
\scalebox{0.9}{
\begin{tikzpicture}
\coordinate (center) at (0,0);
\def\radius{3.0cm}
\draw (center) circle[radius=\radius];
\draw[fill] (center) circle[radius=2pt];
\draw (center) node[below] {\small $z_0$};
\draw (center) -- ++(90:\radius);
\draw (center) -- ++(-30:\radius);
\draw (center) -- ++(-150:\radius);

\foreach \i in {1,2} {
\fill (center) ++(90:\i) circle[radius=2pt] 
	node[left] {\small $u_{\i}$};
\fill (center) ++(-30:\i) circle[radius=2pt]
	node[above] {\small $v_{\i}$};
\fill (center) ++(-150:\i) circle[radius=2pt]
	node[below] {\small $w_{\i}$};
}

\foreach \i in {3,4,5,6,7,8} {
\fill (center) ++(150-20*\i:\radius) circle[radius=2pt];
\draw (center) ++(150-20*\i:\radius+0.8em)	
	node {\small $u_{\i}$};
\fill (center) ++(30-20*\i:\radius) circle[radius=2pt];
\draw (center) ++(30-20*\i:\radius+0.8em)
	node {\small $v_{\i}$};
\fill (center) ++(270-20*\i:\radius) circle[radius=2pt];
\draw (center) ++(270-20*\i:\radius+0.8em)
	node {\small $w_{\i}$};

}

\draw[->] (-0.1,0.2) -- (-0.1,0.8);
\draw (0,0.5) node[left] {\scriptsize $c_2$};
\draw[->] (-0.1,1.2) -- (-0.1,1.8);
\draw (0,1.5) node[left] {\scriptsize $c_1$};
\draw[->] (-0.1,2.2) -- (-0.1,2.8);
\draw (0,2.5) node[left] {\scriptsize $x$};

\draw[->] (0.1,0.8) -- (0.1,0.2);
\draw (0,0.5) node[right] {\scriptsize $a_3$};
\draw[->] (0.1,1.8) -- (0.1,1.2);
\draw (0,1.5) node[right] {\scriptsize $a_4$};
\draw[->] (0.1,2.8) -- (0.1,2.2);
\draw (0.1,2.5) node[right] {\scriptsize $a_5$};

\draw[->] (center) ++(28:\radius-0.3em) arc
			(28:12:\radius-0.3em);
\draw (center) ++(20:\radius-0.7em) node {\scriptsize $a_4$};
\draw[->] (center) ++(8:\radius-0.3em) arc
			(8:-8:\radius-0.3em);
\draw (center) ++(0:\radius-0.7em) node {\scriptsize $a_3$};
\draw[->] (center) ++(-12:\radius-0.3em) arc
			(-12:-28:\radius-0.3em);
\draw (center) ++(-20:\radius-0.7em) node {\scriptsize $a_5$};

\draw[<-] (center) ++(88:\radius+0.3em) arc
			(88:72:\radius+0.3em);
\draw (center) ++(80:\radius+0.8em) node {\scriptsize $b_4$};
\draw[<-] (center) ++(68:\radius+0.3em) arc
			(68:52:\radius+0.3em);
\draw (center) ++(60:\radius+0.8em) node {\scriptsize $b_5$};
\draw[<-] (center) ++(48:\radius+0.3em) arc
			(48:32:\radius+0.3em);
\draw (center) ++(40:\radius+0.8em) node {\scriptsize $b_3$};

\draw[->] (center) ++(152:\radius+0.3em) arc
			(152:168:\radius+0.3em);
\draw (center) ++(160:\radius+0.8em) node {\scriptsize $a_3$};
\draw[->] (center) ++(172:\radius+0.3em) arc
			(172:188:\radius+0.3em);
\draw (center) ++(180:\radius+0.8em) node {\scriptsize $a_5$};
\draw[->] (center) ++(192:\radius+0.3em) arc
			(192:208:\radius+0.3em);
\draw (center) ++(200:\radius+0.8em) node {\scriptsize $a_4$};

\end{tikzpicture}
} 
}
\subfloat[Case 2:\label{fig:c362}] {
\scalebox{0.9}{
\begin{tikzpicture}
\coordinate (center) at (0,0);
\def\radius{3.0cm}
\draw (center) circle[radius=\radius];
\draw[fill] (center) circle[radius=2pt];
\draw (center) node[below] {\small $z_0$};
\draw (center) -- ++(90:\radius);
\draw (center) -- ++(-30:\radius);
\draw (center) -- ++(-150:\radius);

\foreach \i in {1,2} {
\fill (center) ++(90:\i) circle[radius=2pt] 
	node[left] {\small $u_{\i}$};
\fill (center) ++(-30:\i) circle[radius=2pt]
	node[above] {\small $v_{\i}$};
\fill (center) ++(-150:\i) circle[radius=2pt]
	node[below] {\small $w_{\i}$};
}

\foreach \i in {3,4,5,6,7,8} {
\fill (center) ++(150-20*\i:\radius) circle[radius=2pt];
\draw (center) ++(150-20*\i:\radius+0.8em)	
	node {\small $u_{\i}$};
\fill (center) ++(30-20*\i:\radius) circle[radius=2pt];
\draw (center) ++(30-20*\i:\radius+0.8em)
	node {\small $v_{\i}$};
\fill (center) ++(270-20*\i:\radius) circle[radius=2pt];
\draw (center) ++(270-20*\i:\radius+0.8em)
	node {\small $w_{\i}$};

}

\draw[->] (-0.1,0.2) -- (-0.1,0.8);
\draw (0,0.5) node[left] {\scriptsize $c_2$};
\draw[->] (-0.1,1.2) -- (-0.1,1.8);
\draw (0,1.5) node[left] {\scriptsize $c_1$};
\draw[->] (-0.1,2.2) -- (-0.1,2.8);
\draw (0,2.5) node[left] {\scriptsize $x$};

\draw[->] (0.1,0.8) -- (0.1,0.2);
\draw (0,0.5) node[right] {\scriptsize $a_3$};
\draw[->] (0.1,1.8) -- (0.1,1.2);
\draw (0,1.5) node[right] {\scriptsize $a_4$};
\draw[->] (0.1,2.8) -- (0.1,2.2);
\draw (0.1,2.5) node[right] {\scriptsize $a_5$};

\draw[->] (center) ++(28:\radius-0.3em) arc
			(28:12:\radius-0.3em);
\draw (center) ++(20:\radius-0.7em) node {\scriptsize $a_3$};
\draw[->] (center) ++(8:\radius-0.3em) arc
			(8:-8:\radius-0.3em);
\draw (center) ++(0:\radius-0.7em) node {\scriptsize $a_5$};
\draw[->] (center) ++(-12:\radius-0.3em) arc
			(-12:-28:\radius-0.3em);
\draw (center) ++(-20:\radius-0.7em) node {\scriptsize $a_4$};

\draw[<-] (center) ++(88:\radius+0.3em) arc
			(88:72:\radius+0.3em);
\draw (center) ++(80:\radius+0.8em) node {\scriptsize $b_5$};
\draw[<-] (center) ++(68:\radius+0.3em) arc
			(68:52:\radius+0.3em);
\draw (center) ++(60:\radius+0.8em) node {\scriptsize $b_3$};
\draw[<-] (center) ++(48:\radius+0.3em) arc
			(48:32:\radius+0.3em);
\draw (center) ++(40:\radius+0.8em) node {\scriptsize $b_4$};

\draw[->] (center) ++(152:\radius+0.3em) arc
			(152:168:\radius+0.3em);
\draw (center) ++(160:\radius+0.8em) node {\scriptsize $a_4$};
\draw[->] (center) ++(172:\radius+0.3em) arc
			(172:188:\radius+0.3em);
\draw (center) ++(180:\radius+0.8em) node {\scriptsize $a_3$};
\draw[->] (center) ++(192:\radius+0.3em) arc
			(192:208:\radius+0.3em);
\draw (center) ++(200:\radius+0.8em) node {\scriptsize $a_5$};

\end{tikzpicture}
} 
}
\caption{\small Cases for Lemma \ref{lem:dg36}}
\label{fig:cases36}
\end{figure}
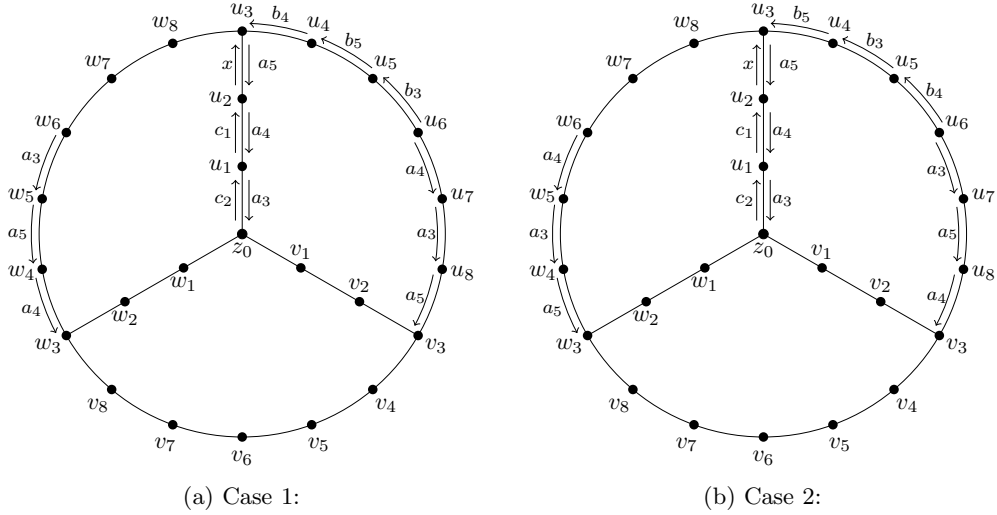

\noindent{\bf Case 1:} $T_{a_5}=w_5\cdots u_3\cdots u_8$. Immediately we
have $x_9=l(\ovr{u_8v_3})=a_5$ and $l(\ovr{w_5w_4})=a_5$, and hence
$y_5=l(\ovr{u_5u_4})=\Phi(l(\ovr{w_5w_4}))=b_5$ (see Figure
\ref{fig:cases36}\subref{fig:c361}). Now $T_{a_4}$
contains the vertex $u_2$ from the path $u_3\cdots u_1$, hence it must
induce an arc of $9$ vertices on outer cycle. By Lemma \ref{lem:treelemma},
 it cannot share its end points with $T_{a_5}$. We see that the only
possibility for $T_{a_4}$ is $u_2u_3\cup w_4\cdots u_3\cdots u_6$.
Similarly we have $T_{a_3}$ as $u_1\cdots u_3\cup w_6\cdots u_3\cdots u_7$.
From these, we conclude as before that $x_8=a_3$, $x_7=a_4$, $y_4=b_4$ and
$y_6=b_3$. From Proposition \ref{prop:pathprop}, we must have
$u_3=\{x_4,\ldots, x_9\}$ and $v_3=\{y_4,\ldots, y_9\}$. Hence we have
$\{x_4,x_5,x_6\}=u_3\backslash \{a_3,a_4,a_5\}$ and
$\{y_7,y_8,y_9\}=v_3\backslash \{b_3,b_4,b_5\}$. Putting $x_3=x\in
\{b_1,b_2\}$, we have the following constraints,
\begin{subequations}
\begin{align}
(x_1,x_2,x_3,x_7,x_8,x_9)=(c_2,c_1,x,a_4,a_3,a_5), &\quad
\{x_4,x_5,x_6\}=u_3\backslash \{a_3,a_4,a_5\}, \\
(y_1,y_2,y_3,y_4,y_5,y_6)=(a_3,a_4,a_5,b_4,b_5,b_3), &\quad
\{y_7,y_8,y_9\}=v_3\backslash \{b_3,b_4,b_5\}.
\end{align}
\end{subequations}
Above equation gives $2$ choices for $x$, $3!=6$ each for $(x_4,x_5,x_6)$
and $(y_7,y_8,y_9)$. We examine the $2\times 6\times 6=72$ cases on computer
using {\tt simpcomp} \cite{simpcomp}. We get the following solution.
\[ X_1=(c_2,c_1,b_1,b_2,a_2,a_1,a_4,a_3,a_5),\quad
Y_1=(a_3,a_4,a_5,b_4,b_5,b_3,c_2,b_1,b_2). \]
The pair $(X_1,Y_1)$ yeilds the complex $N_1$.
\smallskip

\noindent{\bf Case 2:} $T_{a_5}=w_4w_5\cdots u_3\cdots u_7$. In this case,
we have $T_{a_4}=u_2u_3\cup w_6\cdots u_3\cdots u_8$ and
$T_{a_3}=u_1\cdots u_3\cup w_5\cdots u_3\cdots u_6$ (see Figure
\ref{fig:cases36}\subref{fig:c362}). Analyzing as in Case
1, we have the following constraints,
\begin{subequations}
\begin{align}
(x_1,x_2,x_3,x_7,x_8,x_9)=(c_2,c_1,x,a_3,a_5,a_4), &\quad 
\{x_4,x_5,x_6\}=u_3\backslash \{a_3,a_4,a_5\}, \\
(y_1,y_2,y_3,y_4,y_5,y_6)=(a_3,a_4,a_5,b_5,b_3,b_4), &\quad
\{y_7,y_8,y_9\}=v_3\backslash \{b_3,b_4,b_5\}.
\end{align}
\end{subequations}
where $x_3=x\in \{b_1,b_2\}$. Examining the possible $72$ cases using {\tt
simpcomp} \cite{simpcomp}, we find no member of $\lKfiven$. Thus $N_1$ is the only minimal
element of ${\cal C}$ with $\Lambda(M)\cong G_{3,6}$. This proves the
lemma.
\end{proof}

\begin{lemma}\label{lem:dg45}
Let $M\in {\cal C}$ be minimal with $\Lambda(M)\cong G_{4,5}$. Then $M\cong
N_2,N_3$ or $N_4$. 
\end{lemma}
\begin{proof}
Let $X=(x_1,x_2,\ldots,x_9)$ and $Y=(y_1,y_2,\ldots,y_9)$ be the tuple
associated with $M$. We claim the following:
\begin{enumerate}[{\rm (a)}]
\item $(x_1,x_2)=(c_2,c_1)$.
\item $(y_1,y_2,y_3)=(a_3,a_4,a_5)$.
\item $(x_3,x_4)\in \{(b_1,a_2),(b_1,b_2),(b_2,a_1),(b_2,b_1)\}$.
\item $y_4\in \{b_4,c_4,b_5,c_5\}$.
\end{enumerate}
We start by proving that $u_4\cup v_4\cup w_4=V$, which would imply that
$u_4$ intersects all $\Phi$-orbits. By Proposition
\ref{prop:degthree}, $z_0\cup u_4\cup v_4\cup w_4=V$. We show that
$z_0\subseteq u_4\cup v_4\cup w_4$. Suppose not, and let $x\in z_0$ be such
that $x\not\in u_4\cup v_4\cup w_4$. Thus $T_x\subseteq
\Lambda(M)-\{u_4,v_4,w_4\}$. Since $|V(T_x)|=10$, we see that
$V(T_x)=\{z_0\}\cup (\bigcup_{i=1}^3\{u_i,v_i,w_i\})$. But then there is at
most one such $x\in V$. Thus $|u_4\cup v_4\cup w_4|\geq 14$. Notcing that
$u_4\cup v_4\cup w_4$ is union of $\Phi$-orbits, we conclude $|u_4\cup
v_4\cup w_4|=15$, and hence $V=u_4\cup v_4\cup w_4$. Thus $u_4$ intersects
all $\Phi$-orbits. 

We now prove (a). The claim $x_1=c_2$ is proved exaclty as in Lemma
\ref{lem:dg36}. We prove $x_2\in \Phi_1$, and then by Lemma
\ref{lem:leaveseq}, Part (d) claim that $x_2=c_1$. Assume that $x_2\not\in \Phi_1$. Then $x_2\in \Phi_2$.
By Lemma \ref{lem:pathlemma2}, $\{x_1,x_2,x_3,x_4\}\cap u_4=\emptyset$.
Since $u_4$ intersects $\Phi_2$, we must have $\Phi_2\not\subseteq
\{x_1,x_2,x_3,x_4\}$. Since $\{x_1,x_2\}\subseteq \Phi_2$, we conclude
$\{x_3,x_4\}\subseteq \Phi_1$. We will show that in this case
$|V(T_{x_3})|>10$. For the purpose of estimating $|V(T_{x_3})|$, assume
$(x_3,x_4)=(a_1,b_1)$. Then looking outwards from $z_0$, $a_1$ leaves along
$\ovr{u_2u_3}$ and $\ovr{w_3w_4}$, and does not leave along the path
$z_0\cdots v_4$. Thus $\{z_0,u_1,u_2,v_1,v_2,v_3,v_4,w_1,w_2,w_3\}$ are
vertices in $T_{a_1}$. Since $v_4$ cannot be a leaf, there must be at
least one more vertex in $T_{a_1}$, and thus number of vertices exceeds
$10$, a contradiction. Note that we would arrive at the same estimate for
any other choice of $(x_3,x_4)$. Thus $x_2\in \Phi_1$ and hence $x_2=c_1$.
This proves (a). 

To prove (b), we claim that $y_3,y_4,y_5$ belong to distinct $\Phi$-orbits
among $\Phi_3,\Phi_4,\Phi_5$. Since distance between the pairs $(u_3,v_2)$
and $(u_3,w_2)$ in $\Lambda(M)$ is $5$, there exist $x,y\in V$ such that
$\{u_3,u_2,u_1,z_0,v_1,v_2\}\subseteq V(T_x)$ and
$\{u_3,u_2,u_1,z_0,w_1,w_2\}\subseteq V(T_y)$. We orient the two trees with
$z_0$ as root (edges oriented towards the root). Now if $y_i=y_j$ for some
$1\leq i<j\leq 3$, we see that $y_j$ is repeated as a label leaving $v_i$
or $w_i$ towards $z_0$. Thus Lemma \ref{lem:oriented} is violated in $T_x$
or $T_y$. Hence $y_3,y_4,y_5$ are from distinct orbits, and by Lemma
\ref{lem:leaveseq}, we have $(y_3,y_4,y_5)=(a_3,a_4,a_5)$. This proves (b).

To prove (c), we prove that $x_3=b_1$ if $x_3\in \Phi_1$ and $x_3=b_2$ if
$x_3\in \Phi_2$. The proof is exaclty the same as of Claim (c) in Lemma \ref{lem:dg36}. Together
with part (a), and the observation that $\Phi_i\not\subseteq
\{x_1,x_2,x_3,x_4\}$ for $i=1,2$, we have $(x_3,x_4)\in
\{(b_1,a_2),(b_1,b_2),(b_2,a_1),(b_2,b_1)\}$. This proves (c).

To prove (d), it is sufficient to show that $y_4\not\in \Phi_3$. Suppose
$y_4\in \Phi_3$, say $y_4=b_3$. Then observe that $l(\ovr{w_4w_3})=a_3$.
Since $w_4$ and $u_1$ are at a distance $5$ in $\Lambda(M)$, there exists
$z\in V$ such that $\{w_4,w_3,w_2,w_1,z_0,u_1\}\subseteq V(T_z)$. Orient $T_z$ with $z_0$ as the
root (edges oriented towards the root). Then
$l(\ovr{u_1z_0})=l(\ovr{w_4w_3})=a_3$, which contradicts Lemma \ref{lem:oriented}. A
similar contradiction also follows if we assume $y_4=c_3$. Thus $y_4\not\in
\Phi_3$, which proves (d).

Under conditions (a), (b), (c) and (d), we attempt to deduce the arrangement
of trees $T_x$. From (d), we have the following cases:
\medskip

\noindent{\bf Case 1:} $y_4\in \{b_4,c_4\}$. Observe that in this case
$u_4\cap \Phi_3=\{a_3\}$ and $u_4\cap \Phi_5=\{a_5\}$. Consequently,
$a_3\not\in v_4\cup w_4$ and $a_5\not\in v_4\cup w_4$. Thus
$T_{a_3},T_{a_5}\subseteq \Lambda(M)-\{z_0,v_4,w_4\}$. Therefore the trees
$T_{a_3}$ and $T_{a_5}$ are confined to the union of the paths $u_1\cdots
u_4$ and the arc $w_5w_6\cdots u_4\cdots u_8$. Since $a_3$ leaves along the
edge $\ovr{u_1z_0}$, it contains vertices $u_1,u_2,u_3$. Hence,
$T_{a_3}$ induces an arc with $7$ vertices on the outer cycle. Similarly
$T_{a_5}$ induces an arc with $9$ vertices on the outer cycle. It can be
seen that the only possibilities for $T_{a_3}$ and $T_{a_5}$ are (see Figure
\ref{fig:cases451}),
\begin{eqnarray}\label{eq:ta3a5}
T_{a_3} &=& u_1u_2u_3u_4\cup w_6w_7w_8u_4u_5u_6u_7, \nonumber \\
T_{a_5} &=& u_3u_4\cup w_5w_6w_7w_8u_4u_5u_6u_7u_8.
\end{eqnarray}

\begin{figure}[h!]
\centering
\scalebox{0.9}{
\begin{tikzpicture}[scale=0.8]
\coordinate (center) at (0,0);
\def\radius{4.0cm}
\draw (center) circle[radius=\radius];
\draw[fill] (center) circle[radius=2pt];
\draw (center) node[below] {\small $z_0$};
\draw (center) -- ++(90:\radius);
\draw (center) -- ++(-30:\radius);
\draw (center) -- ++(-150:\radius);

\foreach \i in {1,2,3} {
\fill (center) ++(90:\i) circle[radius=2pt] 
	node[left] {\small $u_{\i}$};
\fill (center) ++(-30:\i) circle[radius=2pt]
	node[above] {\small $v_{\i}$};
\fill (center) ++(-150:\i) circle[radius=2pt]
	node[below] {\small $w_{\i}$};
}

\foreach \i in {4,5,6,7,8} {
\fill (center) ++(186-24*\i:\radius) circle[radius=2pt];
\draw (center) ++(186-24*\i:\radius+0.8em)	
	node {\small $u_{\i}$};
\fill (center) ++(66-24*\i:\radius) circle[radius=2pt];
\draw (center) ++(66-24*\i:\radius+0.8em)
	node {\small $v_{\i}$};
\fill (center) ++(306-24*\i:\radius) circle[radius=2pt];
\draw (center) ++(306-24*\i:\radius+0.8em)
	node {\small $w_{\i}$};

}

\draw[->] (-0.1,0.2) -- (-0.1,0.8);
\draw (0,0.5) node[left] {\scriptsize $c_2$};
\draw[->] (-0.1,1.2) -- (-0.1,1.8);
\draw (0,1.5) node[left] {\scriptsize $c_1$};
\draw[->] (-0.1,2.2) -- (-0.1,2.8);
\draw (0,2.5) node[left] {\scriptsize $x$};
\draw[->] (-0.1,3.2) -- (-0.1,3.8);
\draw (0,3.5) node[left] {\scriptsize $y$};

\draw[->] (0.1,0.8) -- (0.1,0.2);
\draw (0,0.5) node[right] {\scriptsize $a_3$};
\draw[->] (0.1,1.8) -- (0.1,1.2);
\draw (0,1.5) node[right] {\scriptsize $a_4$};
\draw[->] (0.1,2.8) -- (0.1,2.2);
\draw (0.1,2.5) node[right] {\scriptsize $a_5$};
\draw[->] (0.1,3.8) -- (0.1,3.2);
\draw (0.1,3.5) node[right] {\scriptsize $z$};

\draw[->] (center) ++(16:\radius-0.3em) arc
			(16:-4:\radius-0.3em);
\draw (center) ++(6:\radius-0.8em) node {\scriptsize $a_3$};
\draw[->] (center) ++(-8:\radius-0.3em) arc
			(-8:-28:\radius-0.3em);
\draw (center) ++(-18:\radius-0.8em) node {\scriptsize $a_5$};

\draw[<-] (center) ++(88:\radius+0.3em) arc
			(88:68:\radius+0.3em);
\draw (center) ++(78:\radius+0.8em) node {\scriptsize $b_5$};
\draw[<-] (center) ++(64:\radius+0.3em) arc
			(64:44:\radius+0.3em);
\draw (center) ++(54:\radius+0.8em) node {\scriptsize $b_3$};

\draw[<-] (center) ++(208:\radius+0.3em) arc
			(208:188:\radius+0.3em);
\draw (center) ++(198:\radius+0.8em) node {\scriptsize $a_5$};
\draw[<-] (center) ++(184:\radius+0.3em) arc
			(184:164:\radius+0.3em);
\draw (center) ++(174:\radius+0.8em) node {\scriptsize $a_3$};

\end{tikzpicture}
}
\caption{\small Illustration for Lemma \ref{lem:dg45}, Case 1}
\label{fig:cases451}
\end{figure}
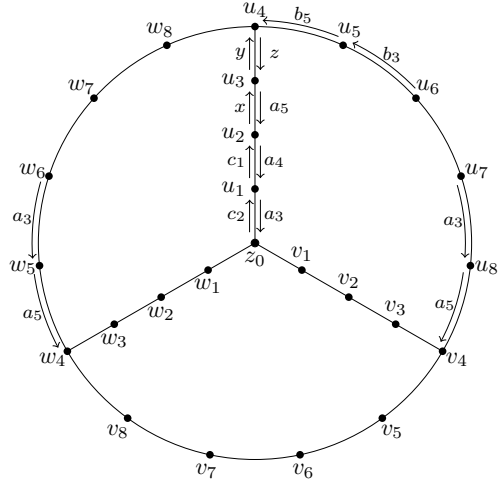

From (\ref{eq:ta3a5}), we conclude $x_8=a_3$, $x_9=a_5$, $y_5=b_5$,
$y_6=b_3$. By Lemma, \ref{lem:pathlemma2}, we have
$\{x_5,\ldots,x_9\}=u_4\backslash v_4$ and
$\{y_5,\ldots,y_9\}=v_4\backslash u_4$. Thus
$\{x_5,x_6,x_7\}=(u_4\backslash v_4)\backslash \{a_3,a_5\}$,
$\{y_7,y_8,y_9\}=(v_4\backslash u_4)\backslash \{b_3,b_5\}$. Putting
$(x_3,x_4)=(x,y)$ and $y_4=z$, we have the following constraints,
\begin{subequations}\label{eq:t451}
\begin{align}
(x_1,x_2,x_3,x_4,x_8,x_9)=(c_2,c_1,x,y,a_3,a_5), &\quad 
\{x_5,x_6,x_7\}=(u_4\backslash v_4)\backslash \{a_3,a_5\}, \\
(y_1,y_2,y_3,y_4,y_5,y_6)=(a_3,a_4,a_5,z,b_3,b_5), &\quad
\{y_7,y_8,y_9\}=(v_4\backslash u_4)\backslash \{b_3,b_5\}.
\end{align}
\end{subequations}
Examining the cases using {\tt simpcomp} \cite{simpcomp}, we find the following solutions:
\begin{align*}
X_1=(c_2,c_1,b_1,b_2,a_4,a_1,a_2,a_3,a_5), &\quad
Y_1=(a_3,a_4,a_5,b_4,b_5,b_3,c_4,b_1,b_2), \\
X_2=(c_2,c_1,b_2,b_1,a_4,a_1,a_2,a_3,a_5), &\quad
Y_2=(a_3,a_4,a_5,b_4,b_5,b_3,c_4,b_2,b_1).
\end{align*}
The tuples $(X_1,Y_1)$ and $(X_2,Y_2)$ yeild the complexes $N_2$ and $N_3$
respectively.
\medskip

\noindent{\bf Case 2:} $y_4\in \{b_5,c_5\}$. In this case we see that the
trees $T_{a_3}$ and $T_{a_4}$ are contained in the union of arc ${\cal A}=
w_5w_6\cdots u_4\cdots u_8$ and the path ${\cal P}=u_1u_2\cdots u_4$. As
$T_{a_4}$ contains $3$ vertices on ${\cal P}$, it must induce a path
containing $8$ vertices (including $u_4$) on ${\cal A}$. Similarly
$T_{a_3}$, which contains $4$ vertices on ${\cal P}$ must induce a path
containing $7$ vertices on ${\cal A}$. It can be seen that we have the
following solutions for $T_{a_3}$ and $T_{a_4}$.
\begin{align}\label{eq:sol45}
T_{a_3} &=u_1u_2u_3u_4\cup w_5w_6w_7w_8u_4u_5u_6, &\quad
T_{a_4} &=u_2u_3u_4\cup w_6w_7w_8u_4u_5u_6u_7u_8, \nonumber \\
T_{a_3} &=u_1u_2u_3u_4\cup w_7w_8u_4u_5u_6u_7u_8, &\quad 
T_{a_4} &=u_2u_3u_4\cup w_5w_6w_7w_8u_4u_5u_6u_7.
\end{align}

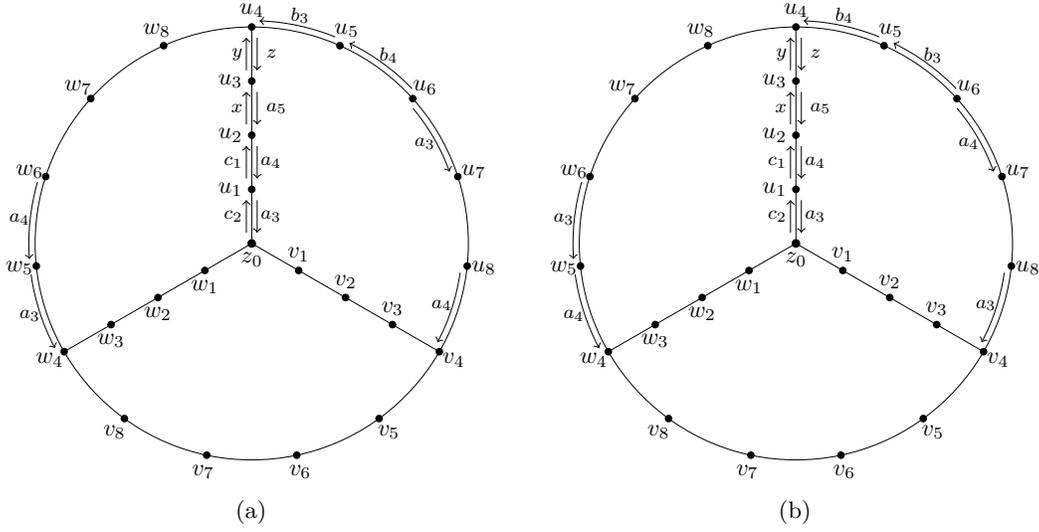
\begin{figure}[htbp]
\centering
\subfloat[\label{fig:case4521}] {
\scalebox{0.9} {
\begin{tikzpicture}[scale=0.8]
\coordinate (center) at (0,0);
\def\radius{4.0cm}
\draw (center) circle[radius=\radius];
\draw[fill] (center) circle[radius=2pt];
\draw (center) node[below] {\small $z_0$};
\draw (center) -- ++(90:\radius);
\draw (center) -- ++(-30:\radius);
\draw (center) -- ++(-150:\radius);

\foreach \i in {1,2,3} {
\fill (center) ++(90:\i) circle[radius=2pt] 
	node[left] {\small $u_{\i}$};
\fill (center) ++(-30:\i) circle[radius=2pt]
	node[above] {\small $v_{\i}$};
\fill (center) ++(-150:\i) circle[radius=2pt]
	node[below] {\small $w_{\i}$};
}

\foreach \i in {4,5,6,7,8} {
\fill (center) ++(186-24*\i:\radius) circle[radius=2pt];
\draw (center) ++(186-24*\i:\radius+0.8em)	
	node {\small $u_{\i}$};
\fill (center) ++(66-24*\i:\radius) circle[radius=2pt];
\draw (center) ++(66-24*\i:\radius+0.8em)
	node {\small $v_{\i}$};
\fill (center) ++(306-24*\i:\radius) circle[radius=2pt];
\draw (center) ++(306-24*\i:\radius+0.8em)
	node {\small $w_{\i}$};

}

\draw[->] (-0.1,0.2) -- (-0.1,0.8);
\draw (0,0.5) node[left] {\scriptsize $c_2$};
\draw[->] (-0.1,1.2) -- (-0.1,1.8);
\draw (0,1.5) node[left] {\scriptsize $c_1$};
\draw[->] (-0.1,2.2) -- (-0.1,2.8);
\draw (0,2.5) node[left] {\scriptsize $x$};
\draw[->] (-0.1,3.2) -- (-0.1,3.8);
\draw (0,3.5) node[left] {\scriptsize $y$};

\draw[->] (0.1,0.8) -- (0.1,0.2);
\draw (0,0.5) node[right] {\scriptsize $a_3$};
\draw[->] (0.1,1.8) -- (0.1,1.2);
\draw (0,1.5) node[right] {\scriptsize $a_4$};
\draw[->] (0.1,2.8) -- (0.1,2.2);
\draw (0.1,2.5) node[right] {\scriptsize $a_5$};
\draw[->] (0.1,3.8) -- (0.1,3.2);
\draw (0.1,3.5) node[right] {\scriptsize $z$};

\draw[->] (center) ++(40:\radius-0.3em) arc
			(40:20:\radius-0.3em);
\draw (center) ++(30:\radius-0.8em) node {\scriptsize $a_3$};
\draw[->] (center) ++(-8:\radius-0.3em) arc
			(-8:-28:\radius-0.3em);
\draw (center) ++(-18:\radius-0.8em) node {\scriptsize $a_4$};

\draw[<-] (center) ++(88:\radius+0.3em) arc
			(88:68:\radius+0.3em);
\draw (center) ++(78:\radius+0.8em) node {\scriptsize $b_3$};
\draw[<-] (center) ++(64:\radius+0.3em) arc
			(64:44:\radius+0.3em);
\draw (center) ++(54:\radius+0.8em) node {\scriptsize $b_4$};

\draw[<-] (center) ++(208:\radius+0.3em) arc
			(208:188:\radius+0.3em);
\draw (center) ++(198:\radius+0.8em) node {\scriptsize $a_3$};
\draw[<-] (center) ++(184:\radius+0.3em) arc
			(184:164:\radius+0.3em);
\draw (center) ++(174:\radius+0.8em) node {\scriptsize $a_4$};

\end{tikzpicture}
} 
}
\subfloat[\label{fig:case4522}] {
\scalebox{0.9} {
\begin{tikzpicture}[scale=0.8]
\coordinate (center) at (0,0);
\def\radius{4.0cm}
\draw (center) circle[radius=\radius];
\draw[fill] (center) circle[radius=2pt];
\draw (center) node[below] {\small $z_0$};
\draw (center) -- ++(90:\radius);
\draw (center) -- ++(-30:\radius);
\draw (center) -- ++(-150:\radius);

\foreach \i in {1,2,3} {
\fill (center) ++(90:\i) circle[radius=2pt] 
	node[left] {\small $u_{\i}$};
\fill (center) ++(-30:\i) circle[radius=2pt]
	node[above] {\small $v_{\i}$};
\fill (center) ++(-150:\i) circle[radius=2pt]
	node[below] {\small $w_{\i}$};
}

\foreach \i in {4,5,6,7,8} {
\fill (center) ++(186-24*\i:\radius) circle[radius=2pt];
\draw (center) ++(186-24*\i:\radius+0.8em)	
	node {\small $u_{\i}$};
\fill (center) ++(66-24*\i:\radius) circle[radius=2pt];
\draw (center) ++(66-24*\i:\radius+0.8em)
	node {\small $v_{\i}$};
\fill (center) ++(306-24*\i:\radius) circle[radius=2pt];
\draw (center) ++(306-24*\i:\radius+0.8em)
	node {\small $w_{\i}$};

}

\draw[->] (-0.1,0.2) -- (-0.1,0.8);
\draw (0,0.5) node[left] {\scriptsize $c_2$};
\draw[->] (-0.1,1.2) -- (-0.1,1.8);
\draw (0,1.5) node[left] {\scriptsize $c_1$};
\draw[->] (-0.1,2.2) -- (-0.1,2.8);
\draw (0,2.5) node[left] {\scriptsize $x$};
\draw[->] (-0.1,3.2) -- (-0.1,3.8);
\draw (0,3.5) node[left] {\scriptsize $y$};

\draw[->] (0.1,0.8) -- (0.1,0.2);
\draw (0,0.5) node[right] {\scriptsize $a_3$};
\draw[->] (0.1,1.8) -- (0.1,1.2);
\draw (0,1.5) node[right] {\scriptsize $a_4$};
\draw[->] (0.1,2.8) -- (0.1,2.2);
\draw (0.1,2.5) node[right] {\scriptsize $a_5$};
\draw[->] (0.1,3.8) -- (0.1,3.2);
\draw (0.1,3.5) node[right] {\scriptsize $z$};

\draw[->] (center) ++(40:\radius-0.3em) arc
			(40:20:\radius-0.3em);
\draw (center) ++(30:\radius-0.8em) node {\scriptsize $a_4$};
\draw[->] (center) ++(-8:\radius-0.3em) arc
			(-8:-28:\radius-0.3em);
\draw (center) ++(-18:\radius-0.8em) node {\scriptsize $a_3$};

\draw[<-] (center) ++(88:\radius+0.3em) arc
			(88:68:\radius+0.3em);
\draw (center) ++(78:\radius+0.8em) node {\scriptsize $b_4$};
\draw[<-] (center) ++(64:\radius+0.3em) arc
			(64:44:\radius+0.3em);
\draw (center) ++(54:\radius+0.8em) node {\scriptsize $b_3$};

\draw[<-] (center) ++(208:\radius+0.3em) arc
			(208:188:\radius+0.3em);
\draw (center) ++(198:\radius+0.8em) node {\scriptsize $a_4$};
\draw[<-] (center) ++(184:\radius+0.3em) arc
			(184:164:\radius+0.3em);
\draw (center) ++(174:\radius+0.8em) node {\scriptsize $a_3$};

\end{tikzpicture}
} 
}
\caption{\small Illustration for Lemma \ref{lem:dg45}, Case 2}
\label{fig:fig452}
\end{figure}

For the first solution in (\ref{eq:sol45}) as illustrated in Figure
\ref{fig:fig452}\subref{fig:case4521}, we have $x_7=a_3$, $x_9=a_4$,
$y_5=b_3$ and $y_6=b_4$. Letting $(x_3,x_4)=(x,y)$ and $y_4=z$ we have the
following constraints,
\begin{subequations}\label{eq:t4521}
\begin{align}
(x_1,x_2,x_3,x_4,x_7,x_9)=(c_2,c_1,x,y,a_3,a_4), &\quad
\{x_5,x_6,x_8\}=(u_4\backslash v_4)\backslash \{a_3,a_4\}, \\
(y_1,y_2,y_3,y_4,y_5,y_6)=(a_3,a_4,a_5,z,b_3,b_4), &\quad
\{y_7,y_8,y_9\}=(v_4\backslash u_4)\backslash \{b_3,b_4\}.
\end{align}
\end{subequations}
Examining the above cases using {\tt simpcomp} \cite{simpcomp}, we get the following
solution,
\begin{align*}
X_3 = (c_2,c_1,b_2,a_1,b_1,c_5,a_3,a_2,a_4), &\quad
Y_3 = (a_3,a_4,a_5,c_5,b_3,b_4,b_5,c_1,b_2).
\end{align*}
The tuple $(X_3,Y_3)$ yeilds the complex $N_4$.

For the second solution in (\ref{eq:sol45}) as illustrated in Figure
\ref{fig:fig452}\subref{fig:case4522}, we have $x_8=a_4$, $x_9=a_3$,
$y_5=b_4$ and $y_7=b_3$. Putting $(x_3,x_4)=(x,y)$ and $y_4=z$, we have the following constraints,
\begin{subequations}\label{eq:t4522}
\begin{align}
(x_1,x_2,x_3,x_4,x_8,x_9)=(c_2,c_1,x,y,a_4,a_3), &\quad
\{x_5,x_6,x_7\}=(u_4\backslash v_4)\backslash \{a_3,a_4\}, \\
(y_1,y_2,y_3,y_4,y_5,y_7)=(a_3,a_4,a_5,z,b_4,b_3), &\quad
\{y_6,y_8,y_9\}=(v_4\backslash u_4)\backslash \{b_3,b_4\}.
\end{align}
\end{subequations}
We do not get any solutions for minimal member of ${\cal C}$ with above
constraints. Thus $N_2,N_3,N_4$ are the only minimal elements of ${\cal C}$
with $G_{4,5}$ as the dual graph.
\end{proof}

\begin{lemma}\label{lem:dg54}
Let $M\in {\cal C}$ be minimal with $\Lambda(M)\cong G_{5,4}$. Then $M\cong
N_5,N_6,\ldots, N_{11}$ or $N_{12}$.
\end{lemma}
\begin{proof}
We begin by proving the following claims:
\begin{enumerate}[{\rm (a)}]
\item $(x_1,x_2,x_3)=(c_2,b_2,c_1)$, $(x_4,x_5)\in
\{(a_1,b_1),(b_1,a_1)\}$.
\item $(y_1,y_2,y_3)=(a_3,a_4,a_5)$, $y_4\in \{b_5,c_5\}$.
\item $y_5\in \{b_3,c_3,b_4,c_4\}$.
\item $y_4=b_5\Rightarrow (x_4,x_5)=(b_1,a_1)$, $y_4=c_5\Rightarrow
(x_4,x_5)=(a_1,b_1)$.
\end{enumerate}
By Lemma \ref{lem:pathlemma2}, we know that
$\{x_1,x_2,\ldots,x_5\}=z_0\backslash u_5$. Since $z_0=\Phi_1\cup \Phi_2$,
we have $\Phi_1\subseteq \{x_1,x_2,\ldots,x_5\}$ or $\Phi_2\subseteq
\{x_1,x_2,\ldots,x_5\}$. By Lemma \ref{lem:leaveseq}, we have $x_1=c_2$.
Let $1\leq p<q<r\leq 5$ be such that $\{x_p,x_q,x_r\}$ is one of the orbits
$\Phi_1$ or $\Phi_2$. Then it can be seen that for $x\in \{x_p,x_q,x_r\}$,
$T_x$ contains $p+q+r-2$ vertices. Thus we must have $p+q+r-2=10$, from
which it follows that $(p,q,r)=(3,4,5)$. Since $x_1=c_2$, $\{x_p,x_q,x_r\}$ must be the orbit
$\Phi_1=\{a_1,b_1,c_1\}$. Hence $x_2\in \Phi_2$. As in Claim (c) in Lemma
\ref{lem:dg36}, we conclude $x_2=b_2$. From Lemma \ref{lem:leaveseq}, we
further have $x_3=c_1$. Therefore $\{x_4,x_5\}=\{a_1,b_1\}$. This proves
(a).

For part (b), $(y_1,y_2,y_3)=(a_3,a_4,a_5)$ follows exactly as in the proof
of Claim (b) in Lemma \ref{lem:dg45}. We now show that $y_4\in \Phi_5$.
Suppose $y_4\not\in \Phi_5$. We show that in this case the trees $T_{c_1}$
and $T_{a_5}$ do not intersect. Since $\Phi_1\subseteq
\{x_1,x_2,\ldots,x_5\}$, we see that $T_{c_1}$ is contained in the union of
the three paths $z_0u_1\cdots u_4$, $z_0v_1\cdots v_4$ and $z_0w_1\cdots
w_4$. Thus $T_{a_5}$ and $T_{c_1}$ must intersect along one of the above
three paths. Clearly $T_{c_1}$ and $T_{a_5}$ do not intersect along the
path $z_0u_1\cdots u_4$. However, if $y_4\not\in \Phi_5$, we see that
$T_{a_5}$ does not contain any vertex on the paths $z_0v_1\cdots v_4$ and
$z_0w_1\cdots w_4$ and hence $T_{a_5}$ cannot intersect $T_{c_1}$, a
contradiction to neighborliness of $M$. Thus $y_4\in \Phi_5$. 

For part (c), it is enough to show that $y_5\not\in \Phi_5$. If $y_5\in
\Phi_5$, then we see that $\Phi_5\subseteq u_5$. Hence $\Phi_5\subseteq
v_5$ and $\Phi_5\subseteq w_5$. But then $|u_5\backslash v_5|\leq 3$. But
this contradicts Lemma \ref{lem:pathlemma2} along the path $u_5u_6\cdots
v_5$, as the four oriented edges
$\ovr{u_5u_6},\ovr{u_6u_7},\ovr{u_7u_8},\ovr{u_8v_5}$ cannot
all have distinct labels. This proves (c).

For part (d), we again use the neighborliness of $M$. Suppose $y_4=b_5$. As
before, the trees $T_{a_1}$ and $T_{b_5}$ must intersect along one of the
paths $z_0u_1\cdots u_4$, $z_0v_1\cdots v_4$ and $z_0w_1\cdots w_4$. We see
that if $y_4=b_4$, $T_{b_5}$ does not contain any vertex from the path
$z_0w_1\cdots w_4$ (As both $z_0$ and $w_5$ do not contain $b_5$). On the
path $z_0v_1\cdots v_4$, $T_{b_5}$ contains $v_4,v_3$, whereas $T_{a_1}$
contains $z_0,v_1,v_2$. Thus $T_{a_1}$ and $T_{b_5}$ must intersect along
$z_0u_1\cdots u_4$. Now $y_4=b_5$, implies only vertex on $z_0u_1\cdots
u_4$ on $T_{b_5}$ is $u_4$. Hence $T_{a_1}$ must also contain $u_4$, and
thus $x_4\neq a_1$. Thus $(x_4,x_5)=(b_1,a_1)$. Similarly if $y_4=c_5$, we
can show that $(x_4,x_5)=(a_1,b_1)$. 
\medskip

\noindent{\bf Case 1:} $y_5\in \{b_3,c_3\}$, $y_4=b_5$. From Claim (d), we
must have $(x_4,x_5)=(b_1,a_1)$ (see Figure \ref{fig:case541}).  Consider
the tree $T_{a_4}$. Notice that $u_5\cap \Phi_4=\{a_4\}$, and hence
$a_4\not\in v_5,w_5$. Thus $T_{a_4}\subseteq \Lambda(M)-\{z_0,v_5,w_5\}$. 
In other words, $T_{a_4}$ is contained
in the union of the path ${\cal P}=z_0u_1\cdots u_5$ and the arc ${\cal
A}=w_6w_7\cdots u_5\cdots u_8$. Since $T_{a_4}$ contains $4$ vertices on
${\cal P}$, it must induce a path containing $7$ vertices (including $u_5$)
on ${\cal A}$. The only possible solution is $T_{a_4}=u_2u_3u_4u_5\cup
w_6w_7w_8u_5u_6u_7u_8$. This gives us $x_9=a_4$, $y_6=b_4$. Putting $y_5=x\in
\{b_3,c_3\}$, we get the following constraints
\begin{subequations}\label{eq:t541}
\begin{align}
(x_1,x_2,x_3,x_4,x_5,x_9)=(c_2,b_2,c_1,b_1,a_1,a_4), &\quad
\{x_6,x_7,x_8\}=(u_5\backslash v_5)\backslash \{a_4\}, \\
(y_1,y_2,y_3,y_4,y_5,y_6)=(a_3,a_4,a_5,b_5,x,b_4), &\quad
\{y_7,y_8,y_9\}=(v_5\backslash u_5)\backslash \{b_4\}.
\end{align}
\end{subequations}

\begin{figure}[htbp]
\centering
\scalebox{0.8}{
\begin{tikzpicture}[scale=0.8]
\coordinate (center) at (0,0);
\def\radius{5.0cm}
\draw (center) circle[radius=\radius];
\draw[fill] (center) circle[radius=2pt];
\draw (center) node[below] {\small $z_0$};
\draw (center) -- ++(90:\radius);
\draw (center) -- ++(-30:\radius);
\draw (center) -- ++(-150:\radius);

\foreach \i in {1,2,3,4} {
\fill (center) ++(90:\i) circle[radius=2pt] 
	node[left] {\small $u_{\i}$};
\fill (center) ++(-30:\i) circle[radius=2pt]
	node[above] {\small $v_{\i}$};
\fill (center) ++(-150:\i) circle[radius=2pt]
	node[below] {\small $w_{\i}$};
}

\foreach \i in {5,6,7,8} {
\fill (center) ++(240-30*\i:\radius) circle[radius=2pt];
\draw (center) ++(240-30*\i:\radius+0.8em)	
	node {\small $u_{\i}$};
\fill (center) ++(120-30*\i:\radius) circle[radius=2pt];
\draw (center) ++(120-30*\i:\radius+0.8em)
	node {\small $v_{\i}$};
\fill (center) ++(360-30*\i:\radius) circle[radius=2pt];
\draw (center) ++(360-30*\i:\radius+0.8em)
	node {\small $w_{\i}$};

}

\draw[->] (-0.1,0.2) -- (-0.1,0.8);
\draw (0,0.5) node[left] {\scriptsize $c_2$};
\draw[->] (-0.1,1.2) -- (-0.1,1.8);
\draw (0,1.5) node[left] {\scriptsize $b_2$};
\draw[->] (-0.1,2.2) -- (-0.1,2.8);
\draw (0,2.5) node[left] {\scriptsize $c_1$};
\draw[->] (-0.1,3.2) -- (-0.1,3.8);
\draw (0,3.5) node[left] {\scriptsize $b_1$};
\draw[->] (-0.1,4.2) -- (-0.1,4.8);
\draw (0,4.5) node[left] {\scriptsize $a_1$};

\draw[->] (0.1,0.8) -- (0.1,0.2);
\draw (0,0.5) node[right] {\scriptsize $a_3$};
\draw[->] (0.1,1.8) -- (0.1,1.2);
\draw (0,1.5) node[right] {\scriptsize $a_4$};
\draw[->] (0.1,2.8) -- (0.1,2.2);
\draw (0.1,2.5) node[right] {\scriptsize $a_5$};
\draw[->] (0.1,3.8) -- (0.1,3.2);
\draw (0.1,3.5) node[right] {\scriptsize $b_5$};
\draw[->] (0.1,4.8) -- (0.1,4.2);
\draw (0.1,4.5) node[right] {\scriptsize $x$};

\draw[->] (center) ++(-4:\radius-0.3em) arc
			(-4:-26:\radius-0.3em);
\draw (center) ++(-15:\radius-0.8em) node {\scriptsize $a_4$};

\draw[<-] (center) ++(86:\radius+0.3em) arc
			(86:64:\radius+0.3em);
\draw (center) ++(75:\radius+0.8em) node {\scriptsize $b_4$};

\draw[<-] (center) ++(206:\radius+0.3em) arc
			(206:184:\radius+0.3em);
\draw (center) ++(195:\radius+0.8em) node {\scriptsize $a_4$};

\end{tikzpicture}
}
\caption{\small Illustration for Lemma \ref{lem:dg54}, Case 1}
\label{fig:case541}
\end{figure}
This case yeilds the following solution when examined using {\tt simpcomp}
\cite{simpcomp}
\begin{align*}
X_1=(c_2,b_2,c_1,b_1,a_1,a_3,a_2,a_5,a_4), &\quad 
Y_1=(a_3,a_4,a_5,b_5,b_3,b_4,b_2,c_3,c_5).
\end{align*}
The pair $(X_1,Y_1)$ yeilds the complex $N_{11}$.
\medskip

\noindent{\bf Case 2:} $y_5\in \{b_3,c_3\}$, $y_4=c_5$. This is similar to
Case 1, except that $y_4=c_5$ and hence from Claim (d),
$(x_4,x_5)=(a_1,b_1)$. Letting $y_5=x\in \{b_3,c_3\}$, we get the
constraints,
\begin{subequations}\label{eq:t542}
\begin{align}
(x_1,x_2,x_3,x_4,x_5,x_9)=(c_2,b_2,c_1,a_1,b_1,a_4), &\quad
\{x_6,x_7,x_8\}=(u_5\backslash v_5)\backslash \{a_4\}, \\
(y_1,y_2,y_3,y_4,y_5,y_6)=(a_3,a_4,a_5,c_5,b_3,b_4), &\quad
\{y_7,y_8,y_9\}=(v_5\backslash u_5)\backslash \{b_4\}.
\end{align}
\end{subequations}
Using {\tt simpcomp} \cite{simpcomp} we get the following solution for this case
\begin{align*}
X_2=(c_2,b_2,c_1,a_1,b_1,a_3,a_2,c_5,a_4), &\quad
Y_2=(a_3,a_4,a_5,c_5,b_3,b_4,b_2,c_3,b_5).
\end{align*}
The pair $(X_2,Y_2)$ gives the complex $N_{12}$.
\medskip

\noindent{\bf Case 3:} $y_5\in \{b_4,c_4\}$, $y_4=b_5$. From Claim (d), we
have $(x_4,x_5)=(b_1,a_1)$. We notice that $u_5\cap \Phi_3=\{a_3\}$, and
hence $a_3\not\in z_0\cup v_3\cup w_3$. Thus $T_{a_3}$ is contained in the
 union of the arc ${\cal A}=w_6w_7\cdots u_5\cdots u_8$ and the path
${\cal P}=u_1u_2\cdots u_5$. Since $T_{a_3}$ contains $5$ vertices on the
path ${\cal P}$, it must induce a $6$ vertex path on ${\cal A}$. We have
the following possiblities for $T_{a_3}$,
\begin{align}\label{eq:possib2}
T_{a_3} &= u_1u_2u_3u_4u_5\cup w_7w_8u_5u_6u_7u_8, \nonumber \\
T_{a_3} &= u_1u_2u_3u_4u_5\cup w_6w_7w_8u_5u_6u_7.
\end{align}

\begin{figure}[htbp]
\centering
\subfloat[Case 3.1\label{fig:cases5431}] {
\scalebox{0.8}{
\begin{tikzpicture}[scale=0.8]
\coordinate (center) at (0,0);
\def\radius{5.0cm}
\draw (center) circle[radius=\radius];
\draw[fill] (center) circle[radius=2pt];
\draw (center) node[below] {\small $z_0$};
\draw (center) -- ++(90:\radius);
\draw (center) -- ++(-30:\radius);
\draw (center) -- ++(-150:\radius);

\foreach \i in {1,2,3,4} {
\fill (center) ++(90:\i) circle[radius=2pt] 
	node[left] {\small $u_{\i}$};
\fill (center) ++(-30:\i) circle[radius=2pt]
	node[above] {\small $v_{\i}$};
\fill (center) ++(-150:\i) circle[radius=2pt]
	node[below] {\small $w_{\i}$};
}

\foreach \i in {5,6,7,8} {
\fill (center) ++(240-30*\i:\radius) circle[radius=2pt];
\draw (center) ++(240-30*\i:\radius+0.8em)	
	node {\small $u_{\i}$};
\fill (center) ++(120-30*\i:\radius) circle[radius=2pt];
\draw (center) ++(120-30*\i:\radius+0.8em)
	node {\small $v_{\i}$};
\fill (center) ++(360-30*\i:\radius) circle[radius=2pt];
\draw (center) ++(360-30*\i:\radius+0.8em)
	node {\small $w_{\i}$};

}

\draw[->] (-0.1,0.2) -- (-0.1,0.8);
\draw (0,0.5) node[left] {\scriptsize $c_2$};
\draw[->] (-0.1,1.2) -- (-0.1,1.8);
\draw (0,1.5) node[left] {\scriptsize $b_2$};
\draw[->] (-0.1,2.2) -- (-0.1,2.8);
\draw (0,2.5) node[left] {\scriptsize $c_1$};
\draw[->] (-0.1,3.2) -- (-0.1,3.8);
\draw (0,3.5) node[left] {\scriptsize $b_1$};
\draw[->] (-0.1,4.2) -- (-0.1,4.8);
\draw (0,4.5) node[left] {\scriptsize $a_1$};

\draw[->] (0.1,0.8) -- (0.1,0.2);
\draw (0,0.5) node[right] {\scriptsize $a_3$};
\draw[->] (0.1,1.8) -- (0.1,1.2);
\draw (0,1.5) node[right] {\scriptsize $a_4$};
\draw[->] (0.1,2.8) -- (0.1,2.2);
\draw (0.1,2.5) node[right] {\scriptsize $a_5$};
\draw[->] (0.1,3.8) -- (0.1,3.2);
\draw (0.1,3.5) node[right] {\scriptsize $b_5$};
\draw[->] (0.1,4.8) -- (0.1,4.2);
\draw (0.1,4.5) node[right] {\scriptsize $x$};

\draw[->] (center) ++(-4:\radius-0.3em) arc
			(-4:-26:\radius-0.3em);
\draw (center) ++(-15:\radius-0.8em) node {\scriptsize $a_3$};

\draw[<-] (center) ++(56:\radius+0.3em) arc
			(56:34:\radius+0.3em);
\draw (center) ++(45:\radius+0.8em) node {\scriptsize $b_3$};

\draw[<-] (center) ++(176:\radius+0.3em) arc
			(176:152:\radius+0.3em);
\draw (center) ++(165:\radius+0.8em) node {\scriptsize $a_3$};

\end{tikzpicture}
}
}
\subfloat[Case 3.2\label{fig:cases5432}] {
\scalebox{0.8}{
\begin{tikzpicture}[scale=0.8]
\coordinate (center) at (0,0);
\def\radius{5.0cm}
\draw (center) circle[radius=\radius];
\draw[fill] (center) circle[radius=2pt];
\draw (center) node[below] {\small $z_0$};
\draw (center) -- ++(90:\radius);
\draw (center) -- ++(-30:\radius);
\draw (center) -- ++(-150:\radius);

\foreach \i in {1,2,3,4} {
\fill (center) ++(90:\i) circle[radius=2pt] 
	node[left] {\small $u_{\i}$};
\fill (center) ++(-30:\i) circle[radius=2pt]
	node[above] {\small $v_{\i}$};
\fill (center) ++(-150:\i) circle[radius=2pt]
	node[below] {\small $w_{\i}$};
}

\foreach \i in {5,6,7,8} {
\fill (center) ++(240-30*\i:\radius) circle[radius=2pt];
\draw (center) ++(240-30*\i:\radius+0.8em)	
	node {\small $u_{\i}$};
\fill (center) ++(120-30*\i:\radius) circle[radius=2pt];
\draw (center) ++(120-30*\i:\radius+0.8em)
	node {\small $v_{\i}$};
\fill (center) ++(360-30*\i:\radius) circle[radius=2pt];
\draw (center) ++(360-30*\i:\radius+0.8em)
	node {\small $w_{\i}$};

}

\draw[->] (-0.1,0.2) -- (-0.1,0.8);
\draw (0,0.5) node[left] {\scriptsize $c_2$};
\draw[->] (-0.1,1.2) -- (-0.1,1.8);
\draw (0,1.5) node[left] {\scriptsize $b_2$};
\draw[->] (-0.1,2.2) -- (-0.1,2.8);
\draw (0,2.5) node[left] {\scriptsize $c_1$};
\draw[->] (-0.1,3.2) -- (-0.1,3.8);
\draw (0,3.5) node[left] {\scriptsize $b_1$};
\draw[->] (-0.1,4.2) -- (-0.1,4.8);
\draw (0,4.5) node[left] {\scriptsize $a_1$};

\draw[->] (0.1,0.8) -- (0.1,0.2);
\draw (0,0.5) node[right] {\scriptsize $a_3$};
\draw[->] (0.1,1.8) -- (0.1,1.2);
\draw (0,1.5) node[right] {\scriptsize $a_4$};
\draw[->] (0.1,2.8) -- (0.1,2.2);
\draw (0.1,2.5) node[right] {\scriptsize $a_5$};
\draw[->] (0.1,3.8) -- (0.1,3.2);
\draw (0.1,3.5) node[right] {\scriptsize $b_5$};
\draw[->] (0.1,4.8) -- (0.1,4.2);
\draw (0.1,4.5) node[right] {\scriptsize $x$};

\draw[->] (center) ++(26:\radius-0.3em) arc
			(26:4:\radius-0.3em);
\draw (center) ++(15:\radius-0.8em) node {\scriptsize $a_3$};

\draw[<-] (center) ++(86:\radius+0.3em) arc
			(86:64:\radius+0.3em);
\draw (center) ++(75:\radius+0.8em) node {\scriptsize $b_3$};

\draw[<-] (center) ++(206:\radius+0.3em) arc
			(206:184:\radius+0.3em);
\draw (center) ++(195:\radius+0.8em) node {\scriptsize $a_3$};

\end{tikzpicture}
}
}
\caption{\small Illustration for Lemma \ref{lem:dg54}, Case 3}
\label{fig:fig543}
\end{figure}
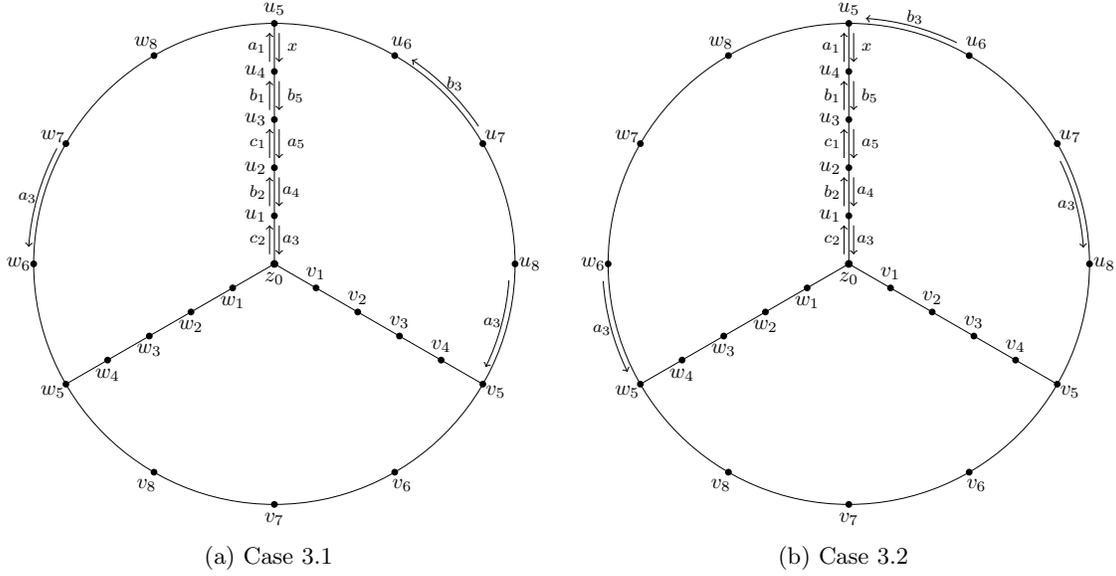

For the first solution in (\ref{eq:possib2}), we get $x_9=a_3$ and
$y_7=b_3$ (see Figure \ref{fig:fig543}\subref{fig:cases5431}). Setting $y_5=x\in \{b_4,c_4\}$, we have the following
constraints,
\begin{subequations}\label{eq:t5431}
\begin{align}
(x_1,x_2,x_3,x_4,x_5,x_9)=(c_2,b_2,c_1,b_1,a_1,a_3), &\quad
\{x_6,x_7,x_8\} = (u_5\backslash v_5)\backslash \{a_3\}, \\
(y_1,y_2,y_3,y_4,y_5,y_7)=(a_3,a_4,a_5,b_5,x,b_3), &\quad
\{y_6,y_8,y_9\} = (v_5\backslash u_5)\backslash \{b_3\}.
\end{align}
\end{subequations}
However, using {\tt simpcomp} \cite{simpcomp} we observe that the above constraints do not yeild any member of ${\cal C}$.

For the second solution in (\ref{eq:possib2}), we get $x_8=a_3$ and
$y_6=b_3$ (see Figure \ref{fig:fig543}\subref{fig:cases5432}). Setting $y_5=x\in \{b_4,c_4\}$, we have the following
constraints,
\begin{subequations}\label{eq:t5432}
\begin{align}
(x_1,x_2,x_3,x_4,x_5,x_8)=(c_2,b_2,c_1,b_1,a_1,a_3), &\quad
\{x_6,x_7,x_9\}=(u_5\backslash v_5)\backslash \{a_3\}, \\
(y_1,y_2,y_3,y_4,y_5,y_6)=(a_3,a_4,a_5,b_5,x,b_3), &\quad
\{y_7,y_8,y_9\}=(v_5\backslash u_5)\backslash \{b_3\}.
\end{align}
\end{subequations}
Using {\tt simpcomp} \cite{simpcomp} we get the following solutions
\begin{align*}
X_3 = (c_2,b_2,c_1,b_1,a_1,a_4,a_5,a_3,a_2), &\quad 
Y_3 = (a_3,a_4,a_5,b_5,b_4,b_3,c_4,c_5,b_2), \\
X_4 = (c_2,b_2,c_1,b_1,a_1,a_5,a_4,a_3,a_2), &\quad
Y_4 = (a_3,a_4,a_5,b_5,b_4,b_3,c_5,c_4,b_2), \\
X_5 = (c_2,b_2,c_1,b_1,a_1,a_5,c_4,a_3,a_2), &\quad
Y_5 = (a_3,a_4,a_5,b_5,c_4,b_3,c_5,b_4,b_2).
\end{align*}
The pairs $(X_3,Y_3),(X_4,Y_4)$ and $(X_5,Y_5)$ yeild $N_5$, $N_6$ and
$N_9$ respectively.
\medskip

\noindent{\bf Case 4:} $y_5\in \{b_4,c_4\}, y_4=c_5$. From Claim (d), we
must have $(x_4,x_5)=(a_1,b_1)$. The rest of analysis is the same as Case
3, where the possiblities for $T_{a_3}$ are given by (\ref{eq:possib2}).
For the first solution in (\ref{eq:possib2}), we have $x_9=a_3$ and
$y_7=b_3$. Setting $y_5=x\in \{b_4,c_4\}$, we have the constraints,
\begin{subequations}\label{eq:t5441}
\begin{align}
(x_1,x_2,x_3,x_4,x_5,x_9)=(c_2,b_2,c_1,a_1,b_1,a_3), &\quad
\{x_6,x_7,x_8\}=(u_5\backslash v_5)\backslash \{a_3\}, \\
(y_1,y_2,y_3,y_4,y_5,y_7)=(a_3,a_4,a_5,c_5,x,b_3), &\quad
\{y_6,y_8,y_9\}=(v_5\backslash u_5)\backslash \{b_3\}.
\end{align}
\end{subequations}
We obtain no solutions for members of ${\cal C}$ meeting above constraints.

For the second solution for $T_{a_3}$ in (\ref{eq:possib2}), we have
$x_8=a_3$ and $y_6=b_3$. Setting $y_5=x$, we have the constraints,
\begin{subequations}\label{eq:t5442}
\begin{align}
(x_1,x_2,x_3,x_4,x_5,x_8)=(c_2,b_2,c_1,a_1,b_1,a_3), &\quad
\{x_6,x_7,x_9\}=(u_5\backslash v_5)\backslash \{a_3\}, \\
(y_1,y_2,y_3,y_4,y_5,y_6)=(a_3,a_4,a_5,c_5,x,b_3), &\quad
\{y_7,y_8,y_9\}=(v_5\backslash u_5)\backslash \{b_3\}.
\end{align}
\end{subequations}
Using {\tt simpcomp} \cite{simpcomp} we obtain the following solutions for members in
${\cal C}$
\begin{align*}
X_6=(c_2,b_2,c_1,a_1,b_1,a_4,c_5,a_3,a_2), &\quad 
Y_6=(a_3,a_4,a_5,c_5,b_4,b_3,c_4,b_5,b_2), \\
X_7=(c_2,b_2,c_1,a_1,b_1,c_5,a_4,a_3,a_2), &\quad
Y_7=(a_3,a_4,a_5,c_5,b_4,b_3,b_5,c_4,b_2), \\
X_8=(c_2,b_2,c_1,a_1,b_1,c_5,c_4,a_3,a_2), &\quad
Y_8=(a_3,a_4,a_5,c_5,c_4,b_3,b_5,b_4,b_2).
\end{align*}
The pairs $(X_6,Y_6)$, $(X_7,Y_7)$ and $(X_8,Y_8)$ yeild the complexes
$N_7, N_8$ and $N_{10}$, respectively. The above cases complete the proof of
the lemma.
\end{proof}

\begin{lemma}\label{lem:dg63}
If $M\in {\cal C}$ is minimal then, $\Lambda(M)\not\cong G_{6,3}$.
\end{lemma}
\begin{proof}
Assume that  $M$ is a minimal member of ${\cal C}$ with
$\Lambda(M)\cong G_{6,3}$. Let $(X,Y)$ be the pair associated with $M$. Now by Lemma
\ref{lem:oriented}, we have $\{x_1,x_2,x_3,x_4,x_5,x_6\}=z_0=\Phi_1\cup
\Phi_2$. Let $p,q,r$ be such that $\{x_p,x_q,x_r\}=\Phi_1$. Then it can be
seen that $T_{a_1}$ contains $p+q+r-2$ vertices. Similarly let $i,j,k$ be
such that $\{x_i,x_j,x_k\}=\Phi_2$. Now $T_{a_2}$ contains $i+j+k-2$
vertices. But since $T_{a_1},T_{a_2}$ each contain $10$ vertices, we must
have $p+q+r+i+j+k=24$. However, $p,q,r,i,j,k$ is a permutation of
$1,2,\ldots,6$, and hence we must have $p+q+r+i+j+k=21$, a contradiction.
This proves the lemma.
\end{proof}

\begin{proof}[{\bf Proof of Theorem \ref{thm:thm1}}]
The proof follows from Lemmata \ref{lem:graphdual2},
\ref{lem:dg36}, \ref{lem:dg45}, \ref{lem:dg54}, \ref{lem:dg63}.
\end{proof}

\begin{proof}[{\bf Proof of Theorem \ref{thm:thm2}}]
The proof follows from Theorem \ref{thm:thm1} and Proposition
\ref{prop:bijection}.
\end{proof}

\noindent{\bf Acknowledgement:} The author would like to thank Basudeb
Datta for useful comments and suggestions. The author would also like to
thank `IISc Mathematics Initiative' and `UGC Centre for Advanced Study' for
support.


\pagebreak

\section*{Appendix A}
We give a proof for the graph theoretic Lemma \ref{lem:graphs}. For
standard terminology on graphs see \cite{bm08}. For a vertex $v$ in
 a graph $G$, $d_G(v)$ will denote the degree of the vertex $v$ in $G$. For
vertices $u,v$ in $G$, $d_G(u,v)$ will denote the length of a shortest path
between $u$ and $v$ in $G$. For a vertex $a\in V(G)$ and a subset $B$ of
$V(G)$, a path $v_0v_1\cdots v_k$ such that $v_0=a$ and
$\{v_0,v_1,\ldots, v_k\}\cap B=\{v_k\}$ is called an $a$-$B$ path. The
following is an easy consequence of the {\em fan lemma} (cf. \cite[Chapter 9]{bm08}).
\begin{lemma}\label{lem:fan}
Let $G$ be a two connected graph and let $B\subseteq V(G)$. If $a\not\in
B$ and $|B|\geq 2$ then, there exist two $a$-$B$ paths in $G$, which intersect only in $a$.
\end{lemma}

We prove the following generalization of Lemma \ref{lem:graphs}.

\begin{theo}
Let graphs $G_{r,s}$ and $T_{r,s}$ be as defined in Examples $\ref{eg:grs}$
and $\ref{eg:trs}$, respectively. Let $G$ be a two connected graph on $n$
vertices with $n+2$ edges. If ${\rm Aut}(G)\supseteq {\mathbb Z}_3$ then
$G\cong G_{r,s}$ for some $r,s>0$ with $3(r+s)=n+2$ or $G\cong T_{r,s}$ for
some $r,s>0$ with $3r+s=n$.
\end{theo}
\begin{proof}
Let $\varphi$ be an order three automorphism of $G$. Let ${\rm
Fix}(\varphi) = \{v\in V(G): \varphi(v)=v\}$ denote the set of vertices
fixed by the automorphism $\varphi$. Let $T$ be the set of vertices with
degree three or more in $G$. Since $G$ is two connected, we have $d_G(v)\geq
2$ for all $v\in V(G)$. Then from the identity 
\begin{equation}
\sum_{v\in V(G)}d_G(v) =2(n+2)=2n+4, \tag{A1}\label{eq:handshake}
\end{equation}
 it follows that $|T|\leq 4$. We have the following cases:
\medskip

\noindent{\bf Case 1:} $T\not\subseteq {\rm Fix}(\varphi)$. Let $u\in T$ be
such that $\varphi(u)\neq u$. Let $v=\varphi(u)$ and $w=\varphi^2(u)$. As
$\varphi$ is an automorphism of $G$, we have $d_G(u)=d_G(v)=d_G(w)$. Let
$k$ be the degree of $u,v$ and $w$ in $G$. Clearly $k\geq 3$. Now from
(\ref{eq:handshake}), it follows that $k=3$, and there exists $z\not\in
\{u,v,w\}$ with $d_G(z)=3$. Since $\varphi$ orbits are either singleton or
three element subsets and $|T|\leq 4$,  we have $\varphi(z)=z$, or $z\in
{\rm Fix}(\varphi)$. Thus we have, $d_G(z,u)=d_G(z,v)=d_G(z,w)=r$ for some
$r\geq 1$. Let $p_{zu} := zu_1\cdots u_r(=u)$ be a shortest $z$-$u$ path in $G$. Then
$p_{zv} := zv_1\cdots v_r(=v)$ and $p_{zw} := zw_1\cdots w_r(=w)$ are $z$-$v$
and $z$-$w$ paths
respectively, where $v_i=\varphi(u_i)$ and $w_i=\varphi^2(u_i)$ for $1\leq
i\leq r$. As $G$ is two connected, $G' := G-z$ is a connected graph. Note
that $\varphi$ is an automorphism of $G'$. Let
$p_{uv} := u_ru_{r+1}\cdots u_{r+s-1}v_r$ be a shortest $u$-$v$ path
in $G'$, where $s=d_{G'}(u,v)=d_{G'}(v,w)=d_{G'}(w,u)$. Let $p_{vw} :=
v_rv_{r+1}\cdots v_{r+s-1}w_r$ and $p_{wu} := w_rw_{r+1}\cdots
w_{r+s-1}u_r$ where $v_i=\varphi(u_i)$ and $w_i=\varphi^2(u_i)$ for $r\leq
i\leq r+s-1$. We claim the following:
\begin{enumerate}[{\rm (a)}]
\item $p_{zu}\cap p_{zv}=p_{zv}\cap p_{zw}=p_{zw}\cap p_{zu}=\{z\}$.
\item $p_{uv}\cap p_{vw}=\{v\}$, $p_{vw}\cap p_{wu}=\{w\}$ and $p_{wu}\cap
p_{uv}=\{u\}$.
\item $p_{zu}\cap p_{uv}=p_{zu}\cap p_{wu}=\{u\}$, $p_{zv}\cap
p_{uv}=p_{zv}\cap p_{vw}=\{v\}$ and $p_{zw}\cap p_{vw}=p_{zw}\cap
p_{wu}=\{w\}$.
\item $p_{zu}\cap p_{vw}=p_{zv}\cap p_{wu}=p_{zw}\cap p_{uv}=\emptyset$.
\end{enumerate}
We first prove (a). Let $i>0$ be maximum such that $u_i\in p_{zu}\cap
p_{zv}$. Then $u_i=v_j$ for some $1\leq j\leq r$. Since
$i=d_G(z,u_i)=d_G(z,v_j)=j$, we have $i=j$. Because $u_r=u\neq v=v_r$, we
have $i<r$. Further, as $d_G(z,w)=r>i$, we have $u_i\neq w$. Thus $u_i\not\in
\{z,u,v,w\}$, and hence $d_G(u_i)=2$. However by maximality of $i$, we have
$\{v_{i-1},v_{i+1},u_{i+1}\}$ as three distinct neighbors of $u_i$, a
contradiction. Therefore $p_{zu}\cap p_{zv}=\{z\}$, and similar argument
works for other pairs. This proves (a). 
Claim (b) can be proved in a manner similar to Claim (a).

To prove Claim (c), we first show that $p_{zu}\cap p_{uv}=\{u\}$. Clearly
$z\not\in p_{uv}$ as $p_{uv}\subseteq G-z$. Let $0<i<r$ be maximum such that
$u_i$ is a vertex on $p_{uv}$. By Claim (a), $u_i\not\in \{z,u,v,w\}$, and
hence $d_G(u_i)=2$. Again maximality of $i$ implies that $u_{i+1}$ is
distinct from the two neighbors of $u_i$ on $p_{uv}$, a contradiction.
Therefore $p_{zu}\cap p_{uv}=\{u\}$. Similarly, we can show for other
pairs. This proves (c).
Claim (d) can be proved in a manner similar to Claim (c).

Define the subgraph 
\[ H := p_{zu}\cup p_{zv}\cup p_{zw}\cup p_{uv}\cup p_{vw}\cup p_{wu}.\] 
Observe that $d_H(v)=d_G(v)$ for all $v\in V(H)$.
Since $G$ is connected, this implies $G=H$. It can be seen that $G\cong
G_{r,s}$ and $3(r+s)=n+2$.
\medskip

\noindent{\bf Case 2:} $T\subseteq {\rm Fix}(\varphi)$. Let $u_1u_2\cdots u_r$ be a maximal path in $G-{\rm Fix}(\varphi)$. Let $x,y$ be
neighbors of $u_1$ and $u_r$ in $G$ respectively, which are not on the
path (such neighbors exist as $d_G(v)\geq 2$ for all $v\in V(G)$). By maximality
of the path, we conclude that $x,y\in {\rm Fix}(\varphi)$. Let $p_u
:= xu_1\cdots u_ry$. Then,
observe that $p_v := xv_1\cdots v_ry$ and $p_w := xw_1\cdots w_ry$ are also
$x$-$y$ paths in $G$, where $v_i=\varphi(u_i)$ and $w_i=\varphi^2(u_i)$ for $1\leq i\leq r$. We note that $x\neq y$, for otherwise we would have
$d_G(x)=6$ and it can be seen that $G-x$ cannot be connected in this case
($G-x$ would be a graph on $n-1$ vertices with $n+2-6=n-4$ edges). Let $H := p_u\cup p_v\cup p_w$. We claim the
following:
\begin{enumerate}[{\rm (a)}]
\item Paths $p_u,p_v$ and $p_w$ are vertex-independent.
\item If $V(H)=V(G)$ then $G=H+xy$. Further $G\cong T_{r,2}$.
\item If $V(H)\neq V(G)$ then $G=H\cup p$ where $p_u, p_v, p_w$ and $p$ are
vertex-independent $x$-$y$ paths. Further $G\cong T_{r,s}$ where $s$ is the
number of vertices in the path $p$.
\end{enumerate}
We first prove (a). Assume that $p_u$ and $p_v$ intersect in a vertex other
than $x$ or $y$. Let $i$ be maximum such that $u_i\in p_v$. Then $u_i=v_j$
for some $1\leq j\leq r$. Since $u_r\neq v_r$, we have $\min(i,j)<r$.
Without loss assume $i<r$. Since $T\subseteq {\rm Fix}(\varphi)$ and
$u_i\not\in {\rm Fix}(\varphi)$, we have $d_G(u_i)=2$. But we see that
$v_{j-1},v_{j+1},u_{i+1}$ are three distinct neighbors of $u_i$, a
contradiction. Therefore $p_u$ and $p_v$ are vertex-independent. Similarly
we can prove for other pairs. 

To prove (b), assume that $V(H)=V(G)$. From Claim (a), we see that $d_H(v)=d_G(v)=2$ for
$v\in V(G)$, $v\not\in \{x,y\}$. It can be seen that to satisfy
(\eqref{eq:handshake}), we must have $G=H+xy$. In this case,  it is readily
seen that $G\cong T_{r,2}$ and $3r+2=n$. This proves (b).

To prove (c), assume that $V(H)\neq V(G)$. Let $a\in V(G)\backslash V(H)$. By Lemma
\ref{lem:fan}, there exist two $a$-$V(H)$ paths $p_1$ and $p_2$ in $G$ such
that $p_1\cap p_2=\{a\}$. Since $d_G(v)=2$ for all $v\in V(H)\backslash
\{x,y\}$, we conclude that the paths $p_1$ and $p_2$ meet $H$ in vertices
$x$ and $y$. Without loss, let $p_1$ be an $a$-$x$ path and $p_2$ be an
$a$-$y$ path. Then $p=p_1\cup p_2$ is an $x$-$y$ path. Clearly $p$ is
vertex-independent to $p_u,p_v$ and $p_w$. Let $H_1=H\cup p$. Then we
observe that $d_{H_1}(v)=2$ for $v\in V(H_1)\backslash \{x,y\}$ and
$d_{H_1}(x)=d_{H_1}(y)=4$. From \eqref{eq:handshake}, it follows that
$d_{H_1}(v)=d_G(v)$ for all $v\in V(H_1)$. Since $G$ is connected, we have
$G=H_1$. It can be seen that $G\cong T_{r,s}$ in this case, where $s$ is
the number of vertices in $p$. Further we have, $3r+s=n$. This completes the proof of the theorem.
\end{proof}
\end{document}